\documentclass{amsart}

\usepackage{graphicx}
\usepackage[shortlabels]{enumitem}

\usepackage{amssymb}

\usepackage{color}

\usepackage{mathtools}

\usepackage{tikz}
\usepackage{tikz-cd}
\usetikzlibrary{arrows,chains,matrix,positioning,scopes}
\usetikzlibrary{arrows.meta}
\usetikzlibrary{matrix}

\usepackage{hyperref}
\usepackage {cleveref}

\newtheorem{theorem}{Theorem}[section]
\newtheorem{proposition}[theorem]{Proposition}
\newtheorem{lemma}[theorem]{Lemma}

\theoremstyle{definition}
\newtheorem{definition}[theorem]{Definition}

\theoremstyle{remark}
\newtheorem{remark}[theorem]{Remark}
\newtheorem{example}[theorem]{Example}

\numberwithin{equation}{section}

\begin{document}

\title[On the rational homotopy type of intersection spaces]{On the rational homotopy type of intersection spaces}

\author[D.J. Wrazidlo]{Dominik J. Wrazidlo}

\address{Institute of Mathematics for Industry, Kyushu University, Motooka 744, Nishi-ku, Fukuoka 819-0395, Japan}
\email{d-wrazidlo@imi.kyushu-u.ac.jp}





\begin{abstract}
Banagl's method of intersection spaces allows to modify certain types of stratified pseudomanifolds near the singular set in such a way that the rational Betti numbers of the modified spaces satisfy generalized Poincar\'{e} duality in analogy with Goresky-MacPherson's intersection homology.
In the case of one isolated singularity, we show that the duality isomorphism comes from a nondegenerate intersection pairing which depends on the choice of a chain representative of the fundamental class of the regular stratum.
On the technical side, we use piecewise linear polynomial differential forms due to Sullivan to define a suitable commutative cochain algebra model for intersection spaces.
Our construction parallels Banagl's commutative cochain algebra of smooth differential forms modeling intersection space cohomology, and we show that both algebras are weakly equivalent.
\end{abstract}

\maketitle

\section{Introduction}
In this paper, we provide a new and systematic approach to intersection space cohomology that is based on tools of rational homotopy theory.

Intersection spaces are a spatial construction due to Banagl \cite{ban2, ban} that gives access to Poincar\'{e} duality for singular spaces.
The cohomology theory of intersection spaces is neither isomorphic to Goresky-MacPherson's intersection homology \cite{gm, gm2}, nor to Cheeger's $L^{2}$ cohomology of Riemannian pseudomanifolds \cite{che, che2, che3}.
For singular Calabi-Yau $3$-folds, the homology of intersection spaces is known to be related to intersection homology by mirror symmetry (see \cite{ban}).
Intersection space cohomology has been modeled by means of linear algebra \cite{ges}, sheaf theory \cite{bm0, bm, bbm, max, ab}, $L^{2}$ cohomology \cite{bh}, and smooth differential forms \cite{ban4}.
While the first two approaches do not take into account the cup product structure on intersection space cohomology, the latter two are only available for pseudomanifolds equipped with a Thom-Mather smooth stratification.
To overcome all these drawbacks, it seems adequate to use commutative cochain algebra models for intersection spaces as introduced in \cite{klim, klim2} for pseudomanifolds with isolated singularities.

Apart from certain real cochain algebras of smooth differential forms used in \cite{ban4}, commutative models have so far not been employed to derive Poincar\'{e} duality theorems for intersection space cohomology.
Note that Klimczak \cite{klim} obtains duality results by turning intersection spaces themselves into Poincar\'{e} duality spaces via cell attachments.
This idea has been extended by the author to certain pseudomanifolds of stratification depth one in \cite{wra2}.
However, even for product link bundles duality results have only been implemented there under an additional condition on the dimension of the singular set.
\par\bigskip

The purpose of this paper is to construct in analogy with Banagl's de Rham approach \cite{ban4} a nondegenerate intersection pairing on the cohomology of a suitable commutative model for intersection spaces.

Let $\overline{p}$ be a perversity in the sense of intersection homology theory, and let $X$ be a compact topologically stratified pseudomanifold having one isolated singularity $x$ with connected link.
While the construction of the pointed intersection space $I^{\overline{p}}X$ involves choosing a Moore approximation of the link, we show that its augmented commutative models do not depend on this choice in the sense that they determine a unique weak equivalence class $\mathcal{I}^{\overline{p}}X$ (see \Cref{main result B}).
We point out that $\mathcal{I}^{\overline{p}}X$ determines the rational homotopy type of simply connected intersection spaces $I^{\overline{p}}X$ by Quillen-Sullivan's theorem.

When $X$ has a Thom-Mather $C^{\infty}$-stratification, we show in \Cref{de rham theorem} that over the reals, a representative of the class $\mathcal{I}^{\overline{p}}X$ is given by Banagl's augmented commutative cochain algebra $\Omega I^{\ast}_{\overline{p}}(X \setminus \{x\}) \oplus \mathbb{R}$ (see \cite{ban4}) consisting of certain smooth differential forms on the complement $M$ in $X$ of a fixed distinguished neighborhood of $\{x\} \subset X$.
Thus, our result gives a strengthening of the de Rham description of intersection space cohomology obtained previously in \cite{ban4, es}.

Recall from \cite{ban4} that generalized Poincar\'{e} duality for $\Omega I^{\ast}_{\overline{p}}(X \setminus \{x\})$ is realized by a canonical nondegenerate intersection pairing on cohomology that integrates wedge product of forms over the top stratum.
In the setting of rational homotopy theory, we imitate the complex $\Omega I^{\ast}_{\overline{p}}(X \setminus \{x\})$ by the augmentation ideal $AI_{\overline{p}}(X, x)$ of a convenient augmented commutative model $AI_{\overline{p}}(X) \in \mathcal{I}^{\overline{p}}X$.
In analogy with the inclusion $\Omega I^{\ast}_{\overline{p}}(X \setminus \{x\}) \subset \Omega^{\ast}(X \setminus \{x\})$, there is a canonical inclusion $\iota_{\overline{p}} \colon AI_{\overline{p}}(X, x) \rightarrow A_{PL}(M)$ (see \Cref{proposition essential diagrams}(i)) into the commutative cochain algebra of piecewise linear polynomial differential forms on $M$.
In our main theorem below, we also employ the linear form $\int_{\mu} \colon A_{PL}(M) \rightarrow \mathbb{Q}$ (see (\ref{integration over chain})) that integrates polynomial differential forms over the singular simplices of a fixed singular chain representative $\mu$ of the fundamental class of $M$ induced by a fixed orientation on $X$.

\begin{theorem}[Generalized Poincar\'{e} duality]\label{main result}
Fix complementary perversities $\overline{p}$ and $\overline{q}$.
If $X$ is $\mathbb{Q}$-oriented and $n$-dimensional, then multiplication in $A_{PL}(M)$ followed by integration over a normalized singular chain representative $\mu$ of the fundamental class of $M$ induces a nondegenerate bilinear form
\begin{align*}
\int_{\mu} \colon H^{r}(AI_{\overline{p}}(X, x)) \times H^{n-r}(AI_{\overline{q}}(X, x)) \rightarrow \mathbb{Q}, \quad ([\alpha], [\beta]) \mapsto \int_{\mu} \iota_{\overline{p}}(\alpha) \cdot \iota_{\overline{q}}(\beta).
\end{align*}
\end{theorem}


We expect that the methods of this paper apply more generally to pseudomanifolds of higher stratification depth (see Remark \Cref{higher stratification depth}), and can be useful to study commutative models for the relative intersection spaces of Agust\'{i}n-de Bobadilla \cite{ab}.

The paper is structured as follows.
\Cref{Preliminaries} provides preliminaries on cochain algebras.
In \Cref{Integration}, we discuss the background on integration of piecewise linear polynomial differential forms that is relevant to this paper.
Cotruncation of augmented commutative cochain algebras will be studied in \Cref{Cotruncation of commutative cochain algebras}.
In \Cref{A commutative model for intersection spaces}, we construct a commutative model for intersection spaces.
The proof of \Cref{main result} is given in \Cref{proof of main result}.
Finally, the purpose of \Cref{Smooth differential forms} is to relate our commutative model for intersection spaces to smooth differential forms.

\section{Preliminaries on cochain algebras}\label{Preliminaries} 
In this section, we discuss basic facts, and fix necessary notation and sign conventions concerning cochain algebras by following the reference \cite{fht}.

Throughout the paper (with the exception of \Cref{Smooth differential forms}), we work over the ground field $\mathbb{Q}$.

\subsection{Graded vector spaces and complexes}(see \S 3(a) in \cite{fht})
A graded vector space is a family $V = \{V_{i}\}_{i \in \mathbb{Z}}$ of rational vector spaces.
An element $v \in V_{i}$ is called (by abuse of language) an element of $V$ of degree $i$.
A linear map $f \colon V \rightarrow W$ of degree $i$ between graded vector spaces is a family of linear maps $f_{j} \colon V_{j} \rightarrow W_{i+j}$.
It determines graded subspaces $\operatorname{ker} f \subset V$ and $\operatorname{im} f \subset W$ via $(\operatorname{ker} f)_{j} = \operatorname{ker} f_{j}$ and $(\operatorname{im} f)_{j} = \operatorname{im} f_{j-i}$, respectively.
A sequence $U \xrightarrow{f} V \xrightarrow{g} W$ of linear maps is exact at $V$ if $\operatorname{ker} g = \operatorname{im} f$.
A short exact sequence is a sequence $0 \rightarrow U \xrightarrow{f} V \xrightarrow{g} W \rightarrow 0$ which is exact at $U$, $V$ and $W$.

A complex is a graded vector space $V$ together with a differential $d$, that is, a linear map $d \colon V \rightarrow V$ of degree $-1$ such that $d^{2} = 0$.
To a complex $V = (V, d)$ 
we may assign its homology $H(V)$ which is the graded vector space given by the quotient $H(V, d) = \operatorname{ker} d / \operatorname{im} d$.
A morphism of complexes is a linear map $\varphi \colon (V, d) \rightarrow (W, d)$ of degree zero such that $d \varphi = \varphi d$.
It induces $H(\varphi) \colon H(V) \rightarrow H(W)$.
If $H(\varphi)$ is an isomorphism, we call $\varphi$ a quasi-isomorphism, and write $\varphi \colon V \xrightarrow{\simeq} W$.
A short exact sequence of morphisms of complexes, $0 \rightarrow (U, d) \xrightarrow{\alpha} (V, d) \xrightarrow{\beta} (W, d) \rightarrow 0$, induces a long exact homology sequence,
\begin{align}\label{long exact homology sequence}
\dots \rightarrow H_{i}(U) \xrightarrow{H_{i}(\alpha)} H_{i}(V) \xrightarrow{H_{i}(\beta)} H_{i}(W) \xrightarrow{\partial} H_{i-1}(U) \rightarrow \dots,
\end{align}
defined for all $i$.
Here, the connecting homomorphism $\partial$ is defined in the usual way:
if $w \in W$ represents $[w] \in H_{i}(W)$ and if $\beta(v) = w$, then $\partial([w])$ is represented by the unique $u \in U$ such that $\alpha(u) = dv$.
A chain complex is a complex $(V, d)$ with $V = \{V_{n}\}_{n \geq 0}$.

In contrast to the previous paragraph, the use of upper grading notation $V = \{V^{i}\}_{i \in \mathbb{Z}}$ in a complex $(V, d)$ will always mean that $d$ has degree $+1$, and then we call $H(V) = H(V, d)$ the cohomology of $V$.
Note that the analog of the connecting homomorphism $\partial$ in the long exact sequence (\ref{long exact homology sequence}) for cohomology will be of the form $\delta \colon H^{i}(W) \rightarrow H^{i+1}(U)$.
A cochain complex is a complex $(V, d)$ with $V = \{V^{n}\}_{n \geq 0}$.

\subsection{Graded algebras}(see \S 3(b) in \cite{fht})
A graded algebra is a graded vector space $R = \{R_{i}\}_{i \in \mathbb{Z}}$ together with bilinear pairings $R_{i} \times R_{j} \rightarrow R_{i+j}$, $(x, y) \mapsto xy$, which are associative (that is, $(xy)z = x(yz)$ for all $x, y, z \in R$) and have an identity $1 \in R_{0}$ (that is, $1x = x = x1$ for all $x \in R$).
We regard $\mathbb{Q}$ as a graded algebra concentrated in degree $0$.
A morphism $\varphi \colon R \rightarrow S$ of graded algebras is a linear map of degree zero such that $\varphi(xy) = \varphi(x) \varphi(y)$ and $\varphi(1) = 1$.
An augmentation for a graded algebra $R$ is a morphism $\varepsilon \colon R \rightarrow \mathbb{Q}$ of graded algebras, and the inclusion $\operatorname{ker} \varepsilon \hookrightarrow R$ is called the augmentation ideal of $\varepsilon$.
A graded algebra is called augmented if it is equipped with an augmentation.

A derivation of degree $k$ is a linear map $\theta \colon R \rightarrow R$ of degree $k$ such that $\theta(xy) = (\theta x)y + (-1)^{k \operatorname{deg}x}(\theta y)$.
A graded algebra $A$ is commutative if $xy = (-1)^{\operatorname{deg} x \operatorname{deg} y}yx$ for all $x, y \in A$.

\subsection{Differential graded algebras}(see \S 3(c) in \cite{fht})
A differential graded algebra (DGA) is a graded algebra $R$ together with a differential $d$ in $R$ that is a derivation.
Note that $\operatorname{ker} d$ is a subalgebra of $R$, and $\operatorname{im} d$ is an ideal in $\operatorname{ker} d$.
Thus, the homology $H(R, d) = \operatorname{ker} d / \operatorname{im} d$ inherits the structure of a graded algebra.
A morphism of differential graded algebras $f \colon (R, d) \rightarrow (S, d)$ is a morphism of graded algebras satisfying $fd = df$.
It induces a morphism $H(f) \colon H(R) \rightarrow H(S)$ of graded algebras.
If $H(f)$ is an isomorphism, we call $f$ a quasi-isomorphism, and write $f \colon (R, d) \xrightarrow{\cong} (S, d)$.
A chain algebra is a DGA $(R, d)$ with $R = \{R_{n}\}_{n \geq 0}$.

A cochain algebra is a DGA $(R, d)$ with $R = \{R^{n}\}_{n \geq 0}$.
For $(R, d)$ we will mainly use the notation $R$ or $R^{\ast}$ in this paper.
As before, the use of upper grading notation means that $d$ has degree $+1$, and the graded algebra $H(R) = H(R, d)$ is called the cohomology of $R$.
A cochain algebra is commutative if this holds for the underlying graded algebra.


\subsection{Normalized singular (co)chains}(see \S 4(a) and \S 5 in \cite{fht})
We recall the concept of normalized singular (co)chains, which will be used in \Cref{de rham}, but might be less familiar than singular (co)chains.

For $n \geq 0$ let $S_{n}(X)$ denote the set of all singular $n$-simplices on a space $X$, that is, continuous maps $\Delta^{n} \rightarrow X$, where $\Delta^{n}$ denotes the convex hull of the standard basis in $\mathbb{R}^{n+1}$.
The $i$-th face inclusion $\lambda_{i} \colon \Delta^{n-1} \rightarrow \Delta^{n}$ of $\Delta^{n}$ (defined for $n \geq 1$ and $0 \leq i \leq n$) and the $j$-th degeneracy $\rho_{j} \colon \Delta^{n+1} \rightarrow \Delta^{n}$ of $\Delta^{n}$ (defined for $n \geq 0$ and $0 \leq j \leq n$) induce the face and degeneracy maps
\begin{align*}
\partial_{i} \colon S_{n+1}(X) \rightarrow S_{n}(X), \qquad \partial_{i}(\sigma) = \sigma \circ f_{i}, \\
s_{j} \colon S_{n}(X) \rightarrow S_{n+1}(X), \qquad s_{j}(\sigma) = \sigma \circ \rho_{j}.
\end{align*}
The singular chain complex of $X$ is the chain complex $CS_{\ast}(X) = \{CS_{n}(X)\}_{n \geq 0}$, where $CS_{n}(X)$ is the rational vector space with basis $S_{n}(X)$, and the differential is given by $d = \sum_{i}(-1)^{i}\partial_{i}$.
Its homology is denoted by $H_{\ast}(X)$, and is called the singular homology of $X$.

Let $DS_{n+1}(X) \subset CS_{n+1}(X)$ denote the subspace spanned by the $(n+1)$-simplices of the form $s_{i}(\tau)$ (degenerate simplices), where $\tau \in S_{n}(X)$ and $0 \leq i \leq n$.
It can be shown that $DS_{\ast}(X)$ is a subchain complex of $CS_{\ast}(X)$, and that $H(DS_{\ast}(X)) = 0$.
The normalized singular chain complex of $X$ is the quotient complex $C_{\ast}(X) = CS_{\ast}(X) / DS_{\ast}(X)$.
As the surjection $CS_{\ast}(X) \rightarrow C_{\ast}(X)$ is a quasi-isomorphism, we may identify $H_{\ast}(X) = H(C_{\ast}(X))$.
A map $f \colon X \rightarrow Y$ induces the complex morphism $f_{\sharp} \colon C_{\ast}(X) \rightarrow C_{\ast}(Y)$ given by $f_{\sharp}(\sigma) = f \circ \sigma$.
We write $f_{\ast} = H(f_{\sharp})$ for the induced map on homology.
For a subspace $A \subset X$ the quotient complex $C_{\ast}(X, A) = C_{\ast}(X) / C_{\ast}(A)$ computes the ordinary relative singular homology $H_{\ast}(X, A) = H(CS_{\ast}(X) / CS_{\ast}(A))$.

The normalized singular cochain algebra of $X$ is the cochain complex $C^{\ast}(X) = \{C^{n}(X)\}_{n \geq 0}$, where $C^{n}(X)$ is the rational vector space of linear forms $\varphi \colon C_{n}(X) \rightarrow \mathbb{Q}$, and the differential is given by the formula $d(\varphi) = -(-1)^{\operatorname{deg} \varphi} \varphi \circ d$.
Moreover, $C^{\ast}(X)$ carries the structure of a cochain algebra by value of the cup product
$$
C^{i}(X) \times C^{i}(X) \rightarrow C^{i+j}(X), \qquad (\varphi, \psi) \mapsto \varphi \cup \psi,
$$
which can be defined in terms of the Alexander-Whitney map.
The cohomology algebra $H(C^{\ast}(X))$ is denoted by $H^{\ast}(X)$ and called the singular cohomology of $X$.

A map $f \colon X \rightarrow Y$ induces the DGA morphism $f^{\sharp} \colon C^{\ast}(Y) \rightarrow C^{\ast}(X)$ given by $f^{\sharp}(\varphi) = \varphi \circ f_{\sharp}$.
We write $f^{\ast} = H(f^{\sharp})$ for the induced map on cohomology.
For any inclusion $i \colon A \rightarrow X$ of a subspace the induced map $i^{\sharp} \colon C^{\ast}(X) \rightarrow C^{\ast}(A)$ is surjective, and its kernel is an ideal of $C^{\ast}(X)$ which will be denoted by $C^{\ast}(X, A)$.
Thus, any pair $(X, A)$ induces a natural short exact sequence
$$
0 \rightarrow C^{\ast}(X, A) \xrightarrow{j^{\sharp}} C^{\ast}(X) \xrightarrow{i^{\sharp}} C^{\ast}(A) \rightarrow 0,
$$
in which we consider the map $j^{\sharp}$ to be induced by the inclusion $j \colon (X, \emptyset) \rightarrow (X, A)$.
More generally, for any map $f \colon X \rightarrow Y$ such that $f(A) \subset B$ for subspaces $A \subset X$ and $B \subset Y$, the morphism $f^{\sharp} \colon C^{\ast}(Y) \rightarrow C^{\ast}(X)$ can be seen to restrict to a morphism $f^{\sharp} \colon C^{\ast}(Y, B) \rightarrow C^{\ast}(X, A)$.

The cohomology algebra of $C^{\ast}(X, A)$ is denoted by $H^{\ast}(X, A)$ and is called the relative singular cohomology of the pair $(X, A)$.


\subsection{The commutative cochain algebra $A_{PL}(X)$}\label{The commutative cochain algebra apl}
(see \S 10(c) in \cite{fht})
Let $X$ be a topological space.
In contrast to the fact that the cochain algebra $C^{\ast}(X)$ of normalized singular cochains is usually not commutative, Sullivan has constructed a contravariant functor $A_{PL}$ from the category of topological spaces and continuous maps to the category of commutative cochain algebras and cochain algebra morphisms such that the graded algebras $H^{\ast}(X)$ and $H(A_{PL}(X))$ are naturally isomorphic (see \Cref{The de Rham theorem}).
In analogy with smooth differential forms on a manifold, elements of $A_{PL}(X)$ are families of polynomial differential forms on the singular simplices of $X$ that are compatible with face and degeneracy maps.

Given a map $f \colon X \rightarrow Y$, we write $f^{\ast} \colon A_{PL}(Y) \rightarrow A_{PL}(X)$ for the induced morphism.
For any inclusion $i \colon A \rightarrow X$ of a subspace the induced map $i^{\ast} \colon A_{PL}(X) \rightarrow A_{PL}(A)$ is surjective, and its kernel is an ideal of $A_{PL}(X)$ which will be denoted by $A_{PL}(X, A)$.
Thus, any pair $(X, A)$ induces a natural short exact sequence
$$
0 \rightarrow A_{PL}(X, A) \xrightarrow{j^{\ast}} A_{PL}(X) \xrightarrow{i^{\ast}} A_{PL}(A) \rightarrow 0,
$$
in which we consider the map $j^{\ast}$ to be induced by the inclusion $j \colon (X, \emptyset) \rightarrow (X, A)$.
More generally, for any map $f \colon X \rightarrow Y$ such that $f(A) \subset B$ for subspaces $A \subset X$ and $B \subset Y$, the morphism $f^{\ast} \colon A_{PL}(Y) \rightarrow A_{PL}(X)$ can be seen to restrict to a morphism $f^{\ast} \colon A_{PL}(Y, B) \rightarrow A_{PL}(X, A)$.


For the one point space $X = \ast$, we have $A_{PL}(\ast) = \mathbb{Q}$.
For a disjoint union $X = X_{1} \sqcup X_{2}$, we can identify $A_{PL}(X) = A_{PL}(X_{1}) \oplus A_{PL}(X_{2})$ by means of the morphisms induced by the inclusions $X_{1}, X_{2} \subset X$.

\section{Integration}\label{Integration}
In this section, we provide the background on integration of piecewise linear polynomial differential forms that is needed to construct the nondegenerate bilinear form of \Cref{main result}.
Our presentation is informed by \S 10(e) in \cite{fht}.
For the convenience of the reader, we go into detail whenever our version of a result is not explicitly stated in \cite{fht}.

Throughout this section, let $(X, A)$ be a pair consisting of a topological space $X$ and subspace $A \subset X$.
Let $i \colon A \rightarrow X$ and $j \colon (X, \emptyset) \rightarrow (X, A)$ denote the inclusions.

\subsection{The de Rham theorem}\label{The de Rham theorem}
As shown in Theorem 10.15(ii) in \cite{fht}, integration of polynomial differential forms over singular simplices of $X$ gives rise to a natural quasi-morphism $\oint_{X} \colon A_{PL}(X) \xrightarrow{\simeq} C^{\ast}(X)$ of cochain complexes.
In the following, we discuss an extension of this result to space pairs.

\begin{theorem}\label{de rham}
For any pair $(X, A)$ there is a quasi-isomorphism
\begin{align}\label{integration}
\oint_{(X, A)} \colon A_{PL}(X, A) \xrightarrow{\simeq} C^{\ast}(X, A)
\end{align}
of cochain complexes.
Moreover, for any second pair $(Y, B)$ and any map $f \colon X \rightarrow Y$ such that $f(A) \subset B$, we obtain a commutative diagram
\begin{equation}\label{diagram naturality of integration}
\begin{tikzcd}
A_{PL}(X, A) \ar{r}{f^{\ast}} \ar{d}{\oint_{(X, A)}} & A_{PL}(Y, B) \ar{d}{\oint_{(Y, B)}} \\
C^{\ast}(X, A) \ar{r}{f^{\sharp}} & C^{\ast}(Y, B).
\end{tikzcd}
%
\end{equation}
\end{theorem}

\begin{proof}
In the case that both $A$ and $B$ are the empty set, our claims are contained in Theorem 10.15(ii) in \cite{fht}.
In the general case, we consider the diagram
\begin{center}
\begin{tikzcd}
0 \ar{r} & A_{PL}(X, A) \ar{r}{j^{\ast}} \ar[d, dashrightarrow, "{\oint_{(X, A)}}"] & A_{PL}(X) \ar{r}{i^{\ast}} \ar{d}{\oint_{X}}[swap]{\simeq} & A_{PL}(A) \ar{r} \ar{d}{\oint_{A}}[swap]{\simeq} & 0 \\
0 \ar{r} & C^{\ast}(X, A) \ar{r}{j^{\sharp}} & C^{\ast}(X) \ar{r}{i^{\sharp}} & C^{\ast}(A) \ar{r} & 0,
\end{tikzcd}
\end{center}
in which the horizontal sequences are the short exact sequences associated to the pair $(X, A)$, and the solid arrow square commutes.
Thus, we see that $\oint_{X}$ restricts to a homomorphism $\oint_{(X, A)} \colon A_{PL}(X, A) \rightarrow C^{\ast}(X, A)$ of cochain complexes.
By the five lemma, $\oint_{(X, A)}$ is a quasi-isomorphism.
Finally, diagram (\ref{diagram naturality of integration}) commutes because it can be obtained by restricting the morphisms in the commutative diagram
\begin{center}
\begin{tikzcd}
A_{PL}(X) \ar{r}{f^{\ast}} \ar{d}{\oint_{X}} & A_{PL}(Y) \ar{d}{\oint_{Y}} \\
C^{\ast}(X) \ar{r}{f^{\sharp}} & C^{\ast}(Y).
\end{tikzcd}
\end{center}
\end{proof}

\subsection{Stokes' theorem}
For a normalized singular chain $\xi \in C_{\ast}(X)$, we use integration \ref{integration}) to define a linear form
\begin{align}\label{integration over chain}
\int_{\xi} \colon A_{PL}(X) \rightarrow \mathbb{Q}, \qquad \int_{\xi}x = (\oint_{X}x)(\xi).
\end{align}

The fact that integration (\ref{integration}) commutes with the differentials can be considered as an abstract version of Stokes' theorem as follows.

\begin{theorem}\label{stokes theorem}
If $\xi \in C_{\ast}(X)$ such that $j_{\sharp}(\xi) \in C_{\ast}(X, A)$ is closed, then $\partial \xi \in C_{\ast}(X)$ is contained in $C_{\ast}(A) \subset C_{\ast}(X)$, and
$$
\int_{\xi} dx = \int_{\partial \xi} i^{\ast}(x), \qquad x \in A_{PL}(X).
$$
\end{theorem}
\begin{proof}
For $\xi \in C_{\ast}(X)$ such that $j_{\sharp}(\xi) \in C_{\ast}(X, A)$ is closed, we have $j_{\sharp}(\partial\xi) = \partial(j_{\sharp}(\xi)) = 0$.
Thus, by exactness of
$$
0 \rightarrow C_{\ast}(A) \xrightarrow{i_{\sharp}} C_{\ast}(X) \xrightarrow{j_{\sharp}} C_{\ast}(X, A) \rightarrow 0,
$$
we may consider $\partial\xi \in C_{\ast}(X)$ as an element of $C_{\ast}(A)$.
Then, for $x \in A_{PL}(X)$ we obtain
\begin{align*}
\int_{\xi}dx &= (\oint_{X}dx)(\xi) = (d \oint_{X}x)(\xi) = (\oint_{X}x)(\partial\xi) \\
&= (\oint_{X}x)(i_{\sharp}(\partial\xi)) = (i^{\sharp}\oint_{X}x)(\partial\xi) = (\oint_{A}i^{\ast} x)(\partial\xi) = \int_{\partial \xi} i^{\ast} x.
\end{align*}
\end{proof}

\subsection{Cohomological multiplicativity}
While integration (\ref{integration}) is not a quasi-isomorphism of cochain \emph{algebras}, it induces an algebra isomorphism on cohomology, which we denote by the same symbol $\oint$.

\begin{theorem}\label{multiplicativity}
For any pair $(X, A)$ we have
\begin{align*}
\oint_{(X, A)}xy &= (\oint_{(X, A)}x) \cup (\oint_{(X, A)}y), \qquad x, y \in H^{\ast}(A_{PL}(X, A)), \\
\oint_{(X, A)}xy &= (\oint_{(X, A)}x) \cup (\oint_{X}y), \qquad x \in H^{\ast}(A_{PL}(X, A)), y \in H^{\ast}(A_{PL}(X)).
\end{align*}
\end{theorem}

\begin{proof}
We only prove the second formula, while the proof of the first formula is very similar.

It follows from the Remark in \cite[p. 130]{fht} that for any pair $(X, A)$ there is a commutative diagram of graded vector spaces
\begin{center}
\begin{tikzcd}
H^{\ast}(X, A) \ar{r}{\beta_{(X, A)}}[swap]{\cong} \ar{dr}[swap]{=} & H^{\ast}((C_{PL} \otimes A_{PL})(X, A)) \ar{d}{\cong}[swap]{\alpha_{(X, A)}} & H^{\ast}(A_{PL}(X, A)) \ar{l}{\cong}[swap]{\gamma_{(X, A)}} \ar{dl}{\oint_{(X, A)}}[swap]{\cong} \\
 & H^{\ast}(X, A), & 
\end{tikzcd}
\end{center}

where we have also used the identification $C_{PL}(-) = C^{\ast}(-)$ from Theorem 10.9(i) in \cite{fht}, as well as the fact that $A_{PL}$, $C_{PL}$ and $C_{PL} \otimes A_{PL}$ are extendable simplicial cochain complexes by Lemma 10.7(iii), Lemma 10.12(ii) and Lemma 10.12(iii) in \cite{fht}, respectively.

Since $\beta \colon C_{PL} \rightarrow C_{PL} \otimes A_{PL}$ and $\gamma \colon A_{PL} \rightarrow C_{PL} \otimes A_{PL}$ are quasi-isomorphisms of extendable simplicial cochain \emph{algebras} (see \cite[p. 125]{fht}), we have
\begin{align*}
\beta_{(X, A)}(x \cup y) &= \beta_{(X, A)}(x) \beta_{X}(y), \qquad x \in H^{\ast}(X, A), y \in H^{\ast}(X), \\
\gamma_{(X, A)}(x y) &= \gamma_{(X, A)}(x) \gamma_{X}(y), \qquad x \in H^{\ast}(A_{PL}(X, A)), y \in H^{\ast}(A_{PL}(X)).
\end{align*}
(Note however that $\alpha \colon C_{PL} \otimes A_{PL} \rightarrow C_{PL}$ is only a quasi-isomorphism of simplicial cochain complexes.)

All in all, we obtain
\begin{align*}
\oint_{(X, A)}xy = \alpha_{(X, A)}(\gamma_{(X, A)}(xy)) = \alpha_{(X, A)}(\gamma_{(X, A)}(x)\gamma_{X}(y)) \\
= \alpha_{(X, A)}(\beta_{(X, A)}(\oint_{(X, A)}x)\beta_{X}(\oint_{X}y)) \\
= \alpha_{(X, A)}(\beta_{(X, A)}((\oint_{(X, A)}x) \cup (\oint_{X}y))) = (\oint_{(X, A)}x) \cup (\oint_{X}y).
\end{align*}
\end{proof}

\subsection{Poincar\'{e}-Lefschetz duality}\label{section poincare lefschetz duality}
Let $(X, A) = (M, \partial M)$ be an $n$-dimensional compact topological manifold with boundary.
We assume that $(M, \partial M)$ is $\mathbb{Q}$-oriented, and denote by $[M, \partial M] \in H_{n}(M, \partial M)$ the corresponding fundamental class.
In the following, we discuss a rational homotopy theory version of the nondegenerate intersection pairing given by classical Poincar\'{e}-Lefschetz duality
\begin{align}\label{poincare lefschetz duality}
H^{r}(M) \times H^{n-r}(M, \partial M) \rightarrow \mathbb{Q}, \qquad (x, y) \mapsto \langle x \cup y, [M, \partial M]\rangle.
\end{align}

\begin{theorem}\label{poincare lefschetz duality theorem}
Let $\mu \in C_{n}(M)$ be a normalized singular chain whose image under the map $j_{\sharp} \colon C_{n}(M) \rightarrow C_{n}(M, \partial M)$
represents the fundamental class $[M, \partial M] \in H_{n}(M, \partial M)$.
Then, multiplication in $A_{PL}(M)$ followed by integration over $\mu$ (see (\ref{integration over chain})) induces a nondegenerate bilinear form
\begin{align*}
\int_{\mu} \colon H^{r}(A_{PL}(M)) \times H^{n-r}(A_{PL}(M, \partial M)) \rightarrow \mathbb{Q}, \qquad ([\alpha], [\beta]) \mapsto \int_{\mu} \alpha \cdot j^{\ast}(\beta).
\end{align*}
\end{theorem}

\begin{proof}
To show that the bilinear form
\begin{align*}
\int_{\mu} \colon A_{PL}^{r}(M) \times A_{PL}^{n-r}(M, \partial M) \rightarrow \mathbb{Q}, \qquad (\alpha, \beta) \mapsto \int_{\mu} \alpha \cdot j^{\ast}(\beta),
\end{align*}
induces a bilinear form $\int_{\mu}$ on cohomology, we have to show that for all closed elements $\alpha, \alpha' \in A_{PL}^{r}(M)$ with $\alpha' - \alpha = d \eta$ for some $\eta \in A_{PL}^{r-1}(M)$, and all closed elements $\beta, \beta' \in A_{PL}^{n-r}(M, \partial M)$ with $\beta' - \beta = d \omega$ for some $\omega \in A_{PL}^{n-r-1}(M, \partial M)$, we have $\int_{\mu}(\alpha', \beta') = \int_{\mu}(\alpha, \beta)$.
It suffices to consider the case that $\alpha = \alpha'$ or $\beta = \beta'$.
If $\beta = \beta'$, then Stokes' theorem (\Cref{stokes theorem}) and the fact that $i^{\ast} \circ j^{\ast} = 0$ imply
\begin{align*}
\int_{\mu}(\alpha', \beta) - \int_{\mu}(\alpha, \beta) = \int_{\mu} d(\eta \cdot j^{\ast}(\beta)) = \int_{\partial\mu} i^{\ast}(\eta \cdot j^{\ast}(\beta)) = 0.
\end{align*}
Similarly, we can show that $\int_{\mu}(\alpha, \beta') = \int_{\mu}(\alpha, \beta)$ if $\alpha = \alpha'$.

We prove that the bilinear form $\int_{\mu}$ induced on cohomology is nondegenerate by reducing it to classical Poincar\'{e}-Lefschetz duality (\ref{poincare lefschetz duality}).
By the de Rham theorem (\Cref{de rham}), we have for $\alpha \in A_{PL}(M)$ and $\beta \in A_{PL}(M, \partial M)$ that
\begin{align*}
\int_{\mu} \alpha \cdot j^{\ast}(\beta) = \int_{\mu} j^{\ast}(\alpha \cdot \beta) = (j^{\sharp} \oint_{(M, \partial M)}\alpha \cdot \beta)(\mu) = (\oint_{(M, \partial M)}\alpha \cdot \beta)(j_{\sharp}\mu),
\end{align*}
where we note that $\alpha \cdot \beta \in A_{PL}(M, \partial M)$ and $j^{\ast}(\alpha \cdot \beta) = \alpha \cdot j^{\ast}(\beta)$, which follows easily from the short exact sequence $0 \rightarrow A_{PL}(M, \partial M) \rightarrow A_{PL}(M) \rightarrow A_{PL}(\partial M) \rightarrow 0$.
Assuming that $\alpha$ and $\beta$ are closed, and using that $j_{\sharp}\mu$ is closed and represents $[j_{\sharp}\mu] = [M, \partial M]$, as well as cohomological multiplicativity (\Cref{multiplicativity}), we can pass to (co)homology and obtain
\begin{align*}
(\oint_{(M, \partial M)}\alpha \cdot \beta)(j_{\sharp}\mu) = \langle [\oint_{(M, \partial M)}\alpha \cdot \beta], [M, \partial M] \rangle \\
= \langle \oint_{(M, \partial M)}\left[\alpha\right] \cdot \left[\beta\right], [M, \partial M] \rangle = \langle \oint_{M}\left[\alpha\right] \cup \oint_{(M, \partial M)}\left[\beta\right], [M, \partial M] \rangle.
\end{align*}
Since $\oint_{M}$ and $\oint_{(M, \partial M)}$ are isomorphisms on cohomology by the de Rham theorem (\Cref{de rham}), we have reduced our bilinear form $\int_{\mu}$ to the nondegenerate bilinear form given by Poincar\'{e}-Lefschetz duality (\ref{poincare lefschetz duality}).
\end{proof}

\section{Cotruncation of commutative cochain algebras}\label{Cotruncation of commutative cochain algebras}
Throughout this section,
let $C^{\ast}$ be a commutative cochain algebra with $H^{0}(C^{\ast}) = \mathbb{Q}$.
The results of this section will be applied in later sections 
to cotruncations of the cochain algebra $C^{\ast} = A_{PL}(L)$ associated to the connected link $L$ of an isolated singularity in pseudomanifold.

\subsection{Truncation and cotruncation}
Let $k>0$ be an integer.
We define the \emph{$k$-truncation cochain complex} of $C^{\ast}$ to be the subcomplex of $C^{\ast}$ given by
$$
\tau_{< k}C^{\ast} \colon \dots \xrightarrow{d} C^{k-2} \xrightarrow{d} C^{k-1} \xrightarrow{d|} \operatorname{im}(d \colon C^{k-1} \rightarrow C^{k}) \rightarrow 0 \rightarrow 0 \rightarrow \dots.
$$
We observe that the canonical inclusion $\vartheta_{<k} \colon \tau_{< k}C^{\ast} \rightarrow C^{\ast}$ induces isomorphisms $H^{r}(\vartheta_{<k}) \colon H^{r}(\tau_{< k}C^{\ast}) \xrightarrow{\cong} H^{r}(C^{\ast})$ for $r < k$, and we have $H^{r}(\tau_{< k}C^{\ast}) = 0$ for $r \geq k$.
However, the inclusion $\tau_{< k}C^{\ast} \rightarrow C^{\ast}$ is in general not a morphism of commutative cochain algebras because $\operatorname{im}(d \colon C^{k-1} \rightarrow C^{k})$ might not be closed under multiplication with elements in $C^{0}$.

In contrast to truncations, it turns out that cotruncations can be studied within the category of commutative cochain algebras.

\begin{definition}\label{definition cotruncation}
A \emph{(cohomological) $k$-cotruncation} of $C^{\ast}$ is a commutative cochain algebra $B^{\ast}$ with $H^{0}(B^{\ast}) = \mathbb{Q}$ together with a morphism $\beta \colon B^{\ast} \rightarrow C^{\ast}$ such that $H^{r}(B^{\ast}) = 0$ for $0 < r < k$, and $H^{r}(\beta) \colon H^{r}(B^{\ast}) \xrightarrow{\cong} H^{r}(C^{\ast})$ for $r \geq k$.
An \emph{augmented $k$-cotruncation} of $C^{\ast}$ is a morphism of the form $(\beta, \varepsilon_{B}) \colon B^{\ast} \rightarrow C^{\ast} \oplus \mathbb{Q}$, where $\beta \colon B^{\ast} \rightarrow C^{\ast}$ is a $k$-cotruncation of $C^{\ast}$.
\end{definition}

The following example shows that cotruncations do always exist.

\begin{example}[Standard cotruncation]\label{example standard cotruncation}
Given a direct sum decomposition $C^{k} = D \oplus \operatorname{im}(d \colon C^{k-1} \rightarrow C^{k})$, we can consider the \emph{$k$-cotruncation cochain complex}
$$
\tau_{\geq k}^{D}C^{\ast} \colon \dots \xrightarrow{} 0 \rightarrow 0 \rightarrow D \xrightarrow{d|} C^{k+1} \xrightarrow{d} C^{k+2} \xrightarrow{d} \dots,
$$
where $D$ is placed in degree $k$.
Using multiplication in $C^{\ast}$, we observe that $\widehat{\tau}_{\geq k}^{D}C^{\ast} = \tau_{\geq k}^{D}C^{\ast} \oplus \mathbb{Q}$ is a commutative cochain algebra with unique augmentation $\varepsilon_{D} \colon \widehat{\tau}_{\geq k}^{D}C^{\ast} \rightarrow \mathbb{Q}$.
Moreover, the canonical inclusion $\vartheta_{\geq k}^{D} \colon \widehat{\tau}_{\geq k}^{D}C^{\ast} \rightarrow C^{\ast}$ defines a $k$-cotruncation of $C^{\ast}$, which we call a \emph{standard $k$-cotruncation} of $C^{\ast}$.
The inclusion $(\vartheta_{\geq k}^{D}, \varepsilon_{D}) \colon \widehat{\tau}_{\geq k}^{D}C^{\ast} \rightarrow C^{\ast} \oplus \mathbb{Q}$ is an augmented $k$-cotruncation of $C^{\ast}$, which we call an \emph{augmented standard $k$-cotruncation} of $C^{\ast}$.

\end{example}

\begin{remark}
While the cotruncation cochain complex $\tau_{\geq k}^{D}C^{\ast}$ in \Cref{example standard cotruncation} is closed under the multiplication inherited from $C^{\ast}$, it might not be an ideal in $C^{\ast}$ because $D$ might not be closed under multiplication with elements in $C^{0}$.
\end{remark}

\subsection{Weak equivalences}\label{Weak equivalences}
Let $A$ be a commutative cochain algebra.
A commutative cochain algebra $B$ is said to be over $A$ if $B$ is equipped with a morphism $B \rightarrow A$.
Moreover, a morphism $f \colon B \rightarrow B'$ of commutative cochain algebras is said to be over $A$ if $B$ and $B'$ are over $A$, and $f$ is compatible with the structure morphisms $B \rightarrow A$ and $B' \rightarrow A$.
Let $f_{1} \colon B_{1} \rightarrow B_{1}'$ and $f_{2} \colon B_{2} \rightarrow B_{2}'$ be morphisms over $A$.
A quasi-isomorphism over $A$ from $f_{1}$ to $f_{2}$
is a pair $(\beta, \beta')$ of quasi-isomorphisms $\beta \colon B_{1} \xrightarrow{\simeq} B_{2}$ and $\beta' \colon B_{1}' \xrightarrow{\simeq} B_{2}'$ over $A$ such that $\beta' \circ f_{1} = f_{2} \circ \beta$.
A weak equivalence over $A$ between $f_{1}$ to $f_{2}$ is a chain $f_{1} \xrightarrow{\simeq} g_{1} \xleftarrow{\simeq} \dots \xrightarrow{\simeq} g_{r} \xleftarrow{\simeq} f_{2}$ of quasi-isomorphisms over $A$, where $g_{1}, \dots, g_{r}$ are suitable morphisms over $A$.
We call $f_{1}$ and $f_{2}$ weakly equivalent over $A$ if they are connected by such a chain.

%
%
%
%

In the following result, we consider augmented cotruncations $B^{\ast} \rightarrow C^{\ast} \oplus \mathbb{Q}$ of $C^{\ast}$ (see \Cref{definition cotruncation}) as morphisms over $C^{\ast} \oplus \mathbb{Q}$ by means of the identity morphism on $C^{\ast} \oplus \mathbb{Q}$.

\begin{proposition}\label{proposition weak equivalence of cotruncations}
Let $k> 0$ be an integer.
Any two augmented $k$-cotruncations of $C^{\ast}$ are weakly equivalent over $C^{\ast} \oplus \mathbb{Q}$.
\end{proposition}

\begin{proof}
For $k=1$, we fix an augmentation $\varepsilon_{C} \colon C^{\ast} \rightarrow \mathbb{Q}$, and show that any given augmented $1$-cotruncation $(\beta, \varepsilon_{B}) \colon B^{\ast} \rightarrow C^{\ast} \oplus \mathbb{Q}$ is weakly equivalent over $C^{\ast} \oplus \mathbb{Q}$ to the morphism $(\operatorname{id}_{C^{\ast}}, \varepsilon_{C}) \colon C^{\ast} \rightarrow C^{\ast} \oplus \mathbb{Q}$.
(Here, $(\operatorname{id}_{C^{\ast}}, \varepsilon_{C})$ is considered as a morphisms over $C^{\ast} \oplus \mathbb{Q}$ by means of the identity morphism on $C^{\ast} \oplus \mathbb{Q}$.)
For this purpose, we choose a quasi-isomorphism $\alpha \colon A^{\ast} \xrightarrow{\simeq} B^{\ast}$ from a commutative cochain algebra $A^{\ast}$ with $A^{0} = \mathbb{Q}$ and unique augmentation $\varepsilon_{A} \colon A^{\ast} \rightarrow \mathbb{Q}$.
(For example, we can use the existence of Sullivan models by Proposition 12.1 in \cite{fht}.)
Then,
\begin{center}
\begin{tikzcd}
B^{\ast} \ar[swap]{d}{(\beta, \varepsilon_{B})} & A^{\ast} \ar{l}{\simeq}[swap]{\alpha} \ar{d}{(\beta \circ \alpha, \varepsilon_{A})} \ar{r}{\beta \circ \alpha} & C^{\ast} \ar{d}{(\operatorname{id}_{C^{\ast}}, \varepsilon_{C})} \\
C^{\ast} \oplus \mathbb{Q} & C^{\ast} \oplus \mathbb{Q} \ar{l}[swap]{=} \ar{r}{=} & C^{\ast} \oplus \mathbb{Q}
\end{tikzcd}
\end{center}
commutes.
This yields the desired weak equivalence over $C^{\ast} \oplus \mathbb{Q}$ between $(\beta, \varepsilon_{B})$ and $(\operatorname{id}_{C^{\ast}}, \varepsilon_{C})$ because $\beta \circ \alpha$ is a $1$-cotruncation and thus a quasi-isomorphism.

For $k > 1$, it suffices to show that
(1) every augmented $k$-cotruncation of $C^{\ast}$ is weakly equivalent over $C^{\ast} \oplus \mathbb{Q}$ to an augmented standard $k$-cotruncation of $C^{\ast}$, and that
(2) any two augmented standard $k$-cotruncations of $C^{\ast}$ are weakly equivalent over $C^{\ast} \oplus \mathbb{Q}$.

\textit{Proof of (1).}
Given any augmented $k$-cotruncation $(\beta, \varepsilon_{B}) \colon B^{\ast} \rightarrow C^{\ast} \oplus \mathbb{Q}$, we choose a quasi-isomorphism $\alpha \colon A^{\ast} \xrightarrow{\simeq} B^{\ast}$ from a commutative cochain algebra $A^{\ast}$ with $A^{0} = \mathbb{Q}$ and unique augmentation $\varepsilon_{A} \colon A^{\ast} \rightarrow \mathbb{Q}$, and satisfying $A^{i} = 0$ for $1 \leq i \leq k-1$, and $d A^{k} = 0$.
(For example, we can take $A^{\ast}$ to be a minimal Sullivan model for $B^{\ast}$ as constructed in Proposition 12.2 in \cite{fht}.)
Since $d A^{k} = 0$ and $A^{k-1} = 0$, we obtain an isomorphism
\begin{center}
\begin{tikzcd}
A^{k} = H^{k}(A^{\ast}) \ar{r}{H^{k}(\alpha)}[swap]{\cong} & H^{k}(B^{\ast}) \ar{r}{H^{k}(\beta)}[swap]{\cong} & H^{k}(C^{\ast}).
\end{tikzcd}
\end{center}
%
Thus, we have $\beta(\alpha(A^{k})) \cap \operatorname{im}(d \colon C^{k-1} \rightarrow C^{k}) = 0$.
Therefore, we may choose a direct sum complement $D \subset C^{k}$ of $\operatorname{im}(d \colon C^{k-1} \rightarrow C^{k})$ in $C^{k}$ such that $\beta(\alpha(A^{k})) \subset D$.
Then, by construction, the $k$-cotruncation $\beta \circ \alpha \colon A^{\ast} \rightarrow C^{\ast}$ lifts to a quasi-isomorphism $\eta \colon A^{\ast} \rightarrow \widehat{\tau}_{\geq k}^{D}C^{\ast}$ under the standard $k$-cotruncation $\vartheta_{\geq k}^{D} \colon \widehat{\tau}_{\geq k}^{D}C^{\ast} \rightarrow C^{\ast}$ of $C^{\ast}$ (see \Cref{example standard cotruncation}).
Consequently, we obtain a commutative diagram
\begin{center}
\begin{tikzcd}
B^{\ast}  \ar{d}[swap]{(\beta, \varepsilon_{B})} & A^{\ast} \ar{l}{\simeq}[swap]{\alpha} \ar{d}{(\beta \circ \alpha, \varepsilon_{A})} \ar{r}{\eta}[swap]{\simeq} & \widehat{\tau}_{\geq k}^{D}C^{\ast} \ar{d}{(\vartheta_{\geq k}^{D}, \varepsilon_{D})} \\
C^{\ast} \oplus \mathbb{Q} & C^{\ast} \oplus \mathbb{Q} \ar{l}[swap]{=} \ar{r}{=} & C^{\ast} \oplus \mathbb{Q}
\end{tikzcd}
\end{center}
which provides a weak equivalence over $C^{\ast} \oplus \mathbb{Q}$ from $(\beta, \varepsilon_{B})$ to $(\vartheta_{\geq k}^{D}, \varepsilon_{D})$.
This completes the proof of assertion (1).

\textit{Proof of (2).}
We fix a subspace $E \subset C^{k-1}$ such that the differential $d \colon C^{k-1} \rightarrow C^{k}$ restricts to an isomorphism $E \xrightarrow {\cong} d(C^{k-1})$.
We consider the cochain complex
$$
\sigma_{\geq k}^{E}C^{\ast} \colon \dots \xrightarrow{} 0 \rightarrow 0 \rightarrow E \xrightarrow{d|} C^{k} \xrightarrow{d} C^{k+1} \xrightarrow{d} \dots,
$$
where $E$ is placed in degree $k-1$.
Since $k > 1$, it follows that $\widehat{\sigma}_{\geq k}^{E}C^{\ast} = \sigma_{\geq k}^{E}C^{\ast} \oplus \mathbb{Q}$ is a commutative cochain algebra whose multiplication is inherited from $C^{\ast}$, and which has a unique augmentation $\varepsilon_{E} \colon \widehat{\sigma}_{\geq k}^{E}C^{\ast} \rightarrow \mathbb{Q}$.
Moreover, the canonical inclusion $\varphi_{\geq k}^{E} \colon \widehat{\sigma}_{\geq k}^{E}C^{\ast} \rightarrow C^{\ast}$ defines a $k$-cotruncation of $C^{\ast}$.

Given any standard $k$-cotruncation $\vartheta_{\geq k}^{D} \colon \widehat{\tau}_{\geq k}^{D}C^{\ast} \rightarrow C^{\ast}$ of $C^{\ast}$, we observe that the canonical inclusion $\rho_{\geq k} \colon \widehat{\tau}_{\geq k}^{D}C^{\ast} \rightarrow \widehat{\sigma}_{\geq k}^{E}C^{\ast}$ is by construction a quasi-isomorphism.
Consequently, we obtain a commutative diagram
\begin{center}
\begin{tikzcd}
\widehat{\tau}_{\geq k}^{D}C^{\ast} \ar{d}[swap]{(\vartheta_{\geq k}^{D}, \varepsilon_{D})} \ar{r}{\rho_{\geq k}}[swap]{\simeq} & \widehat{\sigma}_{\geq k}^{E}C^{\ast} \ar{d}{(\varphi_{\geq k}^{E}, \varepsilon_{E})}  \\
C^{\ast} \oplus \mathbb{Q} \ar{r}{=} & C^{\ast} \oplus \mathbb{Q}
\end{tikzcd}
\end{center}
which provides a quasi-isomorphism over $C^{\ast} \oplus \mathbb{Q}$ from $(\vartheta_{\geq k}^{D}, \varepsilon_{D})$ to $(\varphi_{\geq k}^{E}, \varepsilon_{E})$.
This completes the proof of assertion (2).
\end{proof}

In this paper, \Cref{proposition weak equivalence of cotruncations} will always be applied in connection with the following lemma, where it provides the weak equivalence between $g'$ and $\gamma \circ g$ over $C'$ (see the proofs of \Cref{main result B} and \Cref{de rham theorem}).

\begin{lemma}\label{corollary weak equivalence of fiber products}
We consider two fiber product squares
\begin{center}
\begin{tikzcd}
A \times_{C} B \ar[r] \ar[d, "\pi"] & B \ar[d, "g"] \\
A  \ar[r, twoheadrightarrow, "f"] & C,
\end{tikzcd}
\qquad\qquad
\begin{tikzcd}
A' \times_{C'} B' \ar[r] \ar[d, "\pi'"] & B' \ar[d, "g'"] \\
A'  \ar[r, twoheadrightarrow, "f'"] & C',
\end{tikzcd}
\end{center}
of morphisms over $\mathbb{Q}$, where the morphisms $f$ and $f'$ are surjective.
%
Suppose that $(\alpha, \gamma) \colon f \xrightarrow{\simeq} f'$ is a quasi-isomorphism over $\mathbb{Q}$ such that $g'$ and the composition $\gamma \circ g$ (seen as morphisms over $C'$ by means of the identity morphism on $C'$) are weakly equivalent over $C'$.
Then, $\pi$ and $\pi'$ are weakly equivalent over $\mathbb{Q}$.
If, in addition, $A = A'$ and the quasi-isomorphism $\alpha$ is the identity morphism on $A$, then $\pi$ and $\pi'$ (seen as morphisms over $A$ by means of the identity morphism on $A$) are weakly equivalent over $A$.
\end{lemma}

\begin{proof}
We apply Lemma 13.3 in \cite{fht} to the commutative diagram over $\mathbb{Q}$
\begin{center}
\begin{tikzcd}
A \ar{d}{\simeq}[swap]{\alpha} \ar[r, twoheadrightarrow, "f"] & C \ar{d}{\simeq}[swap]{\gamma} & B \ar{l}[swap]{g} \ar{d}{=} \\
A' \ar[r, twoheadrightarrow, "f'"] & C' & B \ar{l}[swap]{\gamma \circ g}
\end{tikzcd}
\end{center}
to obtain a quasi-isomorphism over $\mathbb{Q}$ from $\pi$ to the structure map $A' \times_{C'} B \rightarrow A'$.
Since $g'$ and the composition $\gamma \circ g$ are by assumption weakly equivalent over $C'$, we can construct a weak equivalence over $\mathbb{Q}$ between $A' \times_{C'} B \rightarrow A'$ and $\pi'$ by applying Lemma 13.3 in \cite{fht} iteratively to commutative diagrams over $\mathbb{Q}$ of the form
\begin{center}
\begin{tikzcd}
A' \ar{d}{}[swap]{=} \ar[r, twoheadrightarrow, "f'"] & C' \ar{d}{}[swap]{=} & B_{0} \ar{l}[swap]{g} \ar{d}{\simeq} \\
A' \ar[r, twoheadrightarrow, "f'"] & C' & B_{1} \ar{l}[swap]{\gamma \circ g}.
\end{tikzcd}
\end{center}
%
%
%

If $A = A'$ and the quasi-isomorphism $\alpha$ is the identity morphism on $A$, then all quasi-isomorphisms we constructed are quasi-isomorphisms over $A$.
\end{proof}

%




\subsection{Poincar\'{e} duality}\label{cotruncation and poincare duality}
For every integer $k > 0$, we fix a standard cotruncation $\vartheta_{\geq k} = \vartheta_{\geq k}^{D} \colon \widehat{\tau}_{\geq k}C^{\ast} = \widehat{\tau}_{\geq k}^{D}C^{\ast} \rightarrow C^{\ast}$ of $C^{\ast}$ (see \Cref{example standard cotruncation}).

\begin{lemma}\label{lemma vanishing products}
If $k, l, r, s > 0$ are integers such that $k+l > r+s$, then
$$
\vartheta_{\geq k}(\tau_{\geq k}C^{r}) \cdot \vartheta_{\geq l}(\tau_{\geq l}C^{s}) = 0.
$$
\end{lemma}

\begin{proof}
Let $\alpha \in \tau_{\geq k}C^{r}$ and $\beta \in \tau_{\geq l}C^{s}$.
By construction of the $k$-cotruncation cochain complex in \Cref{example standard cotruncation}, we have $\vartheta_{\geq k}(\alpha) = 0$ unless $\alpha$ has degree $r \geq k$.
But in the latter case, our assumption $k+l > r+s$ implies that $s < l$, and thus $\vartheta_{\geq l}(\beta) = 0$.
Hence, in all cases we obtain $\vartheta_{\geq k}(\alpha) \cdot \vartheta_{\geq l}(\beta) = 0$.
\end{proof}

For the rest of this section, we assume that $H^{\ast}(C^{\ast})$ is a Poincar\'{e} duality algebra of dimension $c \geq 0$ (see e.g. Definition 3.1 in \cite{fot}).
That is, the vector spaces $H^{r}(C^{\ast})$ are finite dimensional for all $r$, and there is an isomorphism $H^{c}(C^{\ast}) \cong \mathbb{Q}$ such that multiplication in $H^{\ast}(C^{\ast})$ induces for all $r$ a nondegenerate bilinear pairing
\begin{align}\label{poincare duality algebra pairing}
\langle \quad, \quad\rangle \colon H^{r}(C^{\ast}) \times H^{c-r}(C^{\ast}) \rightarrow H^{c}(C^{\ast}) \cong \mathbb{Q}.
\end{align}

\begin{lemma}\label{lemma truncated poincare duality}
If $k, l > 0$ are integers such that $k+l > c$, then the pairing (\ref{poincare duality algebra pairing}) restricts to a nondegenerate pairing
\begin{align}\label{truncated poincare duality algebra pairing}
H^{r}(\tau_{< k}C^{\ast}) \times H^{c-r}(\tau_{\geq l}C^{\ast}) \rightarrow \mathbb{Q}, \quad (x, y) \mapsto \langle H(\vartheta_{<k})(x), H(\vartheta_{\geq l})(y)\rangle.
\end{align}
\end{lemma}

\begin{proof}
If $r \geq k$, then it follows from $k+l > c$ that $c-r < l$, so both vector spaces of the pairing are zero.
For $r < k$, the pairing is isomorphic to the pairing (\ref{poincare duality algebra pairing}).
\end{proof}

Let $L$ be a closed $\mathbb{Q}$-oriented topological manifold of dimension $c$.
Then, $C^{\ast} = A_{PL}(L)$ is a $c$-dimensional Poincar\'{e} duality algebra by means of the isomorphism $H^{c}(C^{\ast}) \cong \mathbb{Q}$ induced by integration $\int_{\lambda} \colon A_{PL}^{c}(L) \rightarrow \mathbb{Q}$ over a chain representative $\lambda \in C_{c}(L; \mathbb{Q})$ of the fundamental class of $L$.
In fact, note that $\int_{\lambda}$ vanishes on $d(A_{PL}^{c-1}(L))$ by Stokes' theorem (\Cref{stokes theorem}), and that the induced pairing $\langle [\alpha], [\beta]\rangle = \int_{\lambda}\alpha \cdot \beta$ (\ref{poincare duality algebra pairing}) is nondegenerate by Poincar\'{e}-Lefschetz duality (\Cref{poincare lefschetz duality theorem}).
Such an integration on cochain level is the structure we need for general $C^{\ast}$ to express the pairing (\ref{truncated poincare duality algebra pairing}) entirely in terms of cotruncations.

\begin{proposition}\label{proposition truncated poincare duality}
Let $k, l > 0$ be integers such that $k+l=c+1$.
Suppose that the isomorphism $H^{c}(C^{\ast}) \cong \mathbb{Q}$ is induced by a linear form $\int \colon C^{c} \rightarrow \mathbb{Q}$ that vanishes on $d(C^{c-1})$.
Then, multiplication in $C^{\ast}$ followed by integration induces a nondegenerate bilinear form
\begin{align*}
\int \colon H^{r}(C^{\ast} / \vartheta_{\geq k}(\tau_{\geq k}C^{\ast})) \times H^{c-r}(\tau_{\geq l}C^{\ast}) \rightarrow \mathbb{Q}, \quad (\left[\pi_{\geq k}(\alpha)\right], \left[\beta\right]) \mapsto \int \alpha \cdot \vartheta_{\geq l}(\beta),
\end{align*}
where $\pi_{\geq k} \colon C^{\ast} \rightarrow C^{\ast}/ \vartheta_{\geq k}(\tau_{\geq k}C^{\ast})$ denotes the canonical projection morphism.
\end{proposition}

\begin{proof}
On cochain level, we have the bilinear form
\begin{align*}
\int \colon C^{r} / \vartheta_{\geq k}(\tau_{\geq k}C^{r}) \times \tau_{\geq l}C^{c-r} \rightarrow \mathbb{Q}, \quad (\pi_{\geq k}(\alpha), \beta) \mapsto \int \alpha \cdot \vartheta_{\geq l}(\beta),
\end{align*}
because for all $\alpha \in C^{r}$, $\eta \in \tau_{\geq k}C^{r}$, and $\beta \in \tau_{\geq l}C^{c-r}$, \Cref{lemma vanishing products} yields that
\begin{align*}
\int(\pi_{\geq k}(\alpha + \vartheta_{\geq k}(\eta)), \beta) - \int(\pi_{\geq k}(\alpha), \beta) = \int \vartheta_{\geq k}(\eta) \cdot \vartheta_{\geq l}(\beta) = 0.
\end{align*}

Let us show that the bilinear form $\int$ induces a bilinear form on cohomology, which we also denote by the symbol $\int$.
For this purpose, we consider closed elements $\pi_{\geq k}(\alpha) \in C^{r} / \tau_{\geq k}C^{r}$ (where $\alpha \in C^{r}$) and $\beta \in \tau_{\geq l}C^{c-r}$.
If $\pi_{\geq k}(\alpha) = d \pi_{\geq k}(\eta)$ for some $\eta \in C^{r-1}$, then
\begin{align*}
\int (\pi_{\geq k}(\alpha), \beta) = \int d\eta \cdot \vartheta_{\geq l}(\beta) = \int d(\eta \cdot \vartheta_{\geq l}(\beta)) = 0.
\end{align*}
If $\beta = d \omega$ for some $\omega \in \tau_{\geq l}C^{c-r-1}$, then we apply \Cref{lemma vanishing products} (using that $d \alpha \in \vartheta_{\geq k}(\tau_{\geq k}C^{r+1})$ because $\pi_{\geq k}(\alpha)$ is closed) to conclude that
\begin{align*}
\int (\pi_{\geq k}(\alpha), \beta) = \int \alpha \cdot \vartheta_{\geq l}(d\omega) = \pm \int d(\alpha \cdot \vartheta_{\geq l}(\omega)) \pm \int d\alpha \cdot \vartheta_{\geq l}(\omega) = 0.
\end{align*}

It remains to show that the bilinear form $\int$ is nondegenerate on cohomology.
For this purpose, we note that the composition
\begin{align*}
\tau_{< k}C^{\ast} \xrightarrow{\vartheta_{<k}} C^{\ast} \xrightarrow{\pi_{\geq k}} C^{\ast} / \vartheta_{\geq k}(\tau_{\geq k}C^{\ast})
\end{align*}
is an isomorphism of cochain complexes because we have the direct sum composition $C^{r} = \tau_{< k}C^{r} \oplus \tau_{\geq k}C^{r}$ by definition of $k$-(co)truncation cochain complexes in \Cref{Cotruncation of commutative cochain algebras}.
But under the induced isomorphism on cohomology, our bilinear form corresponds to the pairing (\ref{truncated poincare duality algebra pairing}), which is nondegenerate by \Cref{lemma truncated poincare duality}.
\end{proof}

\section{A commutative model for intersection spaces}\label{A commutative model for intersection spaces}
Let $X^{n}$ be a compact $n$-dimensional topologically stratified pseudomanifold with one isolated singularity $x$.
In other words, the regular stratum $X \setminus \{x\}$ of $X$ is a topological manifold of dimension $n \geq 2$, and the singular stratum $\{x\}$ of $X$ has an open neighborhood $U_{x}$ in $X$ which is homeomorphic to the open cone on some $(n-1)$-dimensional closed topological manifold $L$, where $x$ corresponds to the cone point.
The manifold $L$ is called the link of $x$, and will be assumed to be connected in this paper.
We fix a distinguished neighborhood $U_{x}$ of $x$ once and for all.
Then, the complement $M = X \setminus U_{x}$ is a compact topological $n$-manifold with boundary $L = \partial M$, and we can write $X = M \cup_{\partial M} \operatorname{cone}(\partial M)$.
%

\subsection{Construction of intersection spaces}
Let $\overline{p}$ be a perversity.
That is, $\overline{p}$ is a function $\{2, 3, 4, \dots\} \rightarrow \{0, 1, 2, \dots\}$ which satisfies the Goresky-MacPherson growth conditions $\overline{p}(2) = 0$ and $\overline{p}(s) \leq \overline{p}(s+1) \leq \overline{p}(s) + 1$ for all $s \in \{2, 3, \dots\}$.
As the singular point $x$ has codimension $n$ in $X$, we have to consider the cutoff degree $k = k(\overline{p}) = n-1 - \overline{p}(n)$ for the link $L = \partial M$ of $x$.
Note that $k > 0$ holds by the growth conditions of $\overline{p}$.
We choose a spatial homology truncation (Moore approximation) $f_{<k} \colon L_{<k} \rightarrow L$ in order to truncate the integral homology groups of the link $L$ in degrees $k$ and above.
That is, $f_{<k}$ induces an isomorphism $H_{r}(L_{<k}; \mathbb{Z}) \xrightarrow{\cong} H_{r}(L; \mathbb{Z})$ for $r < k$, and we have $H_{r}(L_{<k}; \mathbb{Z}) = 0$ for $r \geq k$.
To construct $f_{<k}$, we compose a CW approximation $L' \rightarrow L$ of $L$ with a Moore approximation $L_{<k} \rightarrow L'$ of the CW complex $L'$ by using Corollary 1.4 in \cite{wra}.
(In particular, it is not necessary to assume as in \cite{ban} that the link $L$ is simply connected.)

Recall from \cite{ban} that the perversity $\overline{p}$ intersection space $I^{\overline{p}}X$ associated to $X$ is defined as the homotopy cofiber of the composition
$$
g \colon L_{<k} \xrightarrow{f_{<k}} L = \partial M \hookrightarrow M.
$$
That is, we have
$$
I^{\overline{p}}X = \operatorname{cone}(g) = M \cup_{\partial M = L} \operatorname{cone}(f_{<k}).
$$
We take the cone point $\ast \in \operatorname{cone}(f_{<k})$ as the natural basepoint of the intersection space $I^{\overline{p}}X$.

\begin{remark}\label{higher stratification depth}
Intersection spaces have first been constructed for pseudomanifolds of stratification depth $1$ with trivial link bundle in \cite{ban2, ban}, and with certain twisted link bundles in \cite{bc}.
A construction for some special pseudomanifolds of stratification depth $2$ has been studied in \cite{ban3}.
Recently, Agust\'{i}n and de Bobadilla \cite{ab} have proposed an inductive procedure to construct relative intersection spaces $I^{\overline{p}}(X, \Sigma)$ for pseudomanifolds $X$ with singular set $\Sigma$ of arbitrary stratification depth with compatible link bundles.
\end{remark}

\subsection{The commutative model $AI_{\overline{p}}(X)$}
Let $\overline{p}$ be a perversity.
In the following definition, we introduce the central object $AI_{\overline{p}}(X)$ of this paper, which is an augmented commutative cochain algebra whose construction is based on the choice of a standard $k$-cotruncation of $A_{PL}(L)$.
\Cref{main result B} below shows that $AI_{\overline{p}}(X)$ is an augmented commutative model for the pointed intersection space $I^{\overline{p}}X$.

\begin{definition}\label{definition convenient rational model}
Given an augmented standard $k$-cotruncation
$$
(\vartheta_{\geq k}, \varepsilon) = (\vartheta_{\geq k}^{D}, \varepsilon_{D}) \colon \widehat{\tau}_{\geq k}A_{PL}(L) = \widehat{\tau}_{\geq k}^{D}A_{PL}(L) \hookrightarrow A_{PL}(L) \oplus \mathbb{Q} = A_{PL}(L \sqcup \ast)
$$
of $A_{PL}(L)$ for the cutoff degree $k = n-1 - \overline{p}(n)$ (see \Cref{example standard cotruncation}), we define the commutative cochain algebra $AI_{\overline{p}}(X)$ as the fiber product that completes the fiber product square
\begin{equation}\label{diagram definition of commutative model}
\begin{tikzcd}
AI_{\overline{p}}(X) = A_{PL}(M \sqcup \ast) \times_{A_{PL}(\partial M \sqcup \ast)} \widehat{\tau}_{\geq k}A_{PL}(L) \ar[r, twoheadrightarrow, "\rho_{\overline{p}}"] \ar[d, hook, "\iota_{\overline{p}}"] & \widehat{\tau}_{\geq k}A_{PL}(L) \ar[d, hook, "{(\vartheta_{\geq k}, \varepsilon)}"] \\
A_{PL}(M \sqcup \ast) \ar[r, twoheadrightarrow, "\operatorname{incl}^{\ast}"] & A_{PL}(L \sqcup \ast).
\end{tikzcd}
\end{equation}
%
We equip $AI_{\overline{p}}(X)$ with the canonical augmentation given by the composition
\begin{align}\label{canonical augmentation}
AI_{\overline{p}}(X) \xrightarrow{(\ref{diagram definition of commutative model})} A_{PL}(L \sqcup \ast) \xrightarrow{\operatorname{incl}^{\ast}} A_{PL}(\ast) = \mathbb{Q}.
\end{align}
\end{definition}

\begin{remark}
While $AI_{\overline{p}}(X)$ depends on the choice of the standard cotruncation $\vartheta_{\geq k}^{D}$, different choices lead to commutative cochain algebras that are weakly equivalent over $A_{PL}(M \sqcup \ast) = A_{PL}(M) \oplus \mathbb{Q}$ by the last part of \Cref{corollary weak equivalence of fiber products}.
\end{remark}

The following proposition will be used in the proof of \Cref{main result} (see \Cref{proof of main result}).

\begin{proposition}\label{proposition essential diagrams}
\begin{enumerate}[(i)]
\item Let $AI_{\overline{p}}(X, x)$ denote the augmentation ideal of the augmented commutative cochain algebra $AI_{\overline{p}}(X)$.
Diagram (\ref{diagram definition of commutative model}) restricts (by abuse of notation) to a fiber product square
\begin{equation}\label{diagram restricted fiber product square}
\begin{tikzcd}
AI_{\overline{p}}(X, x) \ar[r, twoheadrightarrow, "\rho_{\overline{p}}"] \ar[d, hook, "\iota_{\overline{p}}"] & \tau_{\geq k}A_{PL}(L) \ar[d, hook, "\vartheta_{\geq k}|"] \\
A_{PL}(M) \ar[r, twoheadrightarrow, "i^{\ast}"] & A_{PL}(L)
\end{tikzcd}
\end{equation}
of commutative cochain complexes, where $i \colon L = \partial M \rightarrow M$ is the inclusion.
\item Let $\eta_{\overline{p}} \colon \operatorname{ker} \rho_{\overline{p}} \hookrightarrow AI_{\overline{p}}(X, x)$ denote the inclusion.
There is a canonical identification $\operatorname{ker} \rho_{\overline{p}} = A_{PL}(M, \partial M)$ such that the diagram
\begin{center}
\begin{tikzcd}
A_{PL}(M, \partial M) \ar[rd, hook, swap, "j^{\ast}"] \ar[r, hook, "\eta_{\overline{p}}"] & AI_{\overline{p}}(X, x) \ar[d, hook, "\iota_{\overline{p}}"] \\
 & A_{PL}(M)
\end{tikzcd}
\end{center}
commutes, where $j \colon (M, \emptyset) \rightarrow (M, \partial M)$ denotes the inclusion of pairs.
\item Let $\kappa_{\overline{p}} \colon A_{PL}(M) \twoheadrightarrow A_{PL}(M) / \operatorname{im} \iota_{\overline{p}}$ denote the quotient map.
There is a canonical identification $A_{PL}(M) / \operatorname{im} \iota_{\overline{p}} = A_{PL}(L) / \vartheta_{\geq k}(\tau_{\geq k}A_{PL}(L))$ such that the diagram
\begin{center}
\begin{tikzcd}
A_{PL}(M) \ar[rd, twoheadrightarrow, swap, "\kappa_{\overline{p}}"] \ar[r, twoheadrightarrow, "i^{\ast}"] & A_{PL}(L) \ar[d, twoheadrightarrow, "\pi_{\geq k}"] \\
 & A_{PL}(L) / \vartheta_{\geq k}(\tau_{\geq k}A_{PL}(L))
\end{tikzcd}
\end{center}
commutes, where $\pi_{\geq k}$ denotes the quotient map.
\end{enumerate}
\end{proposition}

\begin{proof}
The morphisms in diagram (\ref{diagram definition of commutative model}) are compatible with the canonical augmentations $AI_{\overline{p}}(X, x) \rightarrow \mathbb{Q}$ (\ref{canonical augmentation}), $\varepsilon \colon \widehat{\tau}_{\geq k}A_{PL}(L) \rightarrow \mathbb{Q}$, $A_{PL}(M \sqcup \ast) \xrightarrow{\operatorname{incl}^{\ast}} A_{PL}(\ast) = \mathbb{Q}$, and $A_{PL}(L \sqcup \ast) \xrightarrow{\operatorname{incl}^{\ast}} A_{PL}(\ast) = \mathbb{Q}$.
Thus, part (i) follows by restricting diagram (\ref{diagram definition of commutative model}) to the corresponding augmentation ideals.
As for part (ii), we note that $\iota_{\overline{p}}$ restricts under $\eta_{\overline{p}} \colon \operatorname{ker} \rho_{\overline{p}} \hookrightarrow AI_{\overline{p}}(X, x)$ and $j^{\ast}$ to an injection $\operatorname{ker} \rho_{\overline{p}} \hookrightarrow \operatorname{ker} i^{\ast} = A_{PL}(M, \partial M)$.
To show surjectivity, we note that for any $y \in \operatorname{ker} i^{\ast}$ there is $x \in AI_{\overline{p}}(X, x)$ such that $\iota_{\overline{p}}(x) = y$ and $\rho_{\overline{p}}(x) = 0$ in the fiber product square (\ref{diagram restricted fiber product square}).
As for part (iii), we note that $i^{\ast}$ restricts under $\kappa_{\overline{p}} \colon A_{PL}(M) \twoheadrightarrow A_{PL}(M) / \operatorname{im} \iota_{\overline{p}}$ and $\pi_{\geq k}$ to a surjection $A_{PL}(M) / \operatorname{im} \iota_{\overline{p}} \twoheadrightarrow A_{PL}(L) / \vartheta_{\geq k}(\tau_{\geq k}A_{PL}(L))$.
To show injectivity, we observe that for any $y \in A_{PL}(M)$ such that $i^{\ast}(y) \in \vartheta_{\geq k}(\tau_{\geq k}A_{PL}(L))$ there is $x \in AI_{\overline{p}}(X, x)$ such that $\iota_{\overline{p}}(x) = y$ by using the fiber product square (\ref{diagram restricted fiber product square}).
\end{proof}

Recall from \Cref{Weak equivalences} the notion of weak equivalence of morphisms.
The following result implies that $AI_{\overline{p}}(X)$ is a augmented commutative model for the pointed intersection space $I^{\overline{p}}X$.
In particular, the weak equivalence class determined by either of the augmented commutative cochain algebras $AI_{\overline{p}}(X)$ and $A_{PL}(I^{\overline{p}}X)$ does not depend on the choice of the Moore approximation $f_{<k}$ or on the choice of the standard cotruncation $\vartheta_{\geq k}^{D}$.

\begin{theorem}\label{main result B}
The morphism $A_{PL}(I^{\overline{p}}X) \rightarrow A_{PL}(M \sqcup \ast)$ induced by inclusion is weakly equivalent over $A_{PL}(M \sqcup \ast)$ to the morphism $\iota_{\overline{p}} \colon AI_{\overline{p}}(X) \rightarrow A_{PL}(M \sqcup \ast)$.
\end{theorem}

\begin{proof}
We apply Proposition 13.5 in \cite{fht} to the commutative diagram of inclusions
\begin{center}
\begin{tikzcd}
\partial M \sqcup \ast \ar{r} \ar{d} & \operatorname{cone}(f_{<k}) \ar{d} \\
M \sqcup \ast \ar{r}{h} & I^{\overline{p}}X
\end{tikzcd}
\end{center}
to obtain a commutative diagram
\begin{center}
\begin{tikzcd}
A_{PL}(I^{\overline{p}}X) \ar{r}{\simeq} \ar{d}{h^{\ast}} & A_{PL}(M \sqcup \ast) \times_{A_{PL}(\partial M \sqcup \ast)} A_{PL}(\operatorname{cone}(f_{<k})) \ar{d}{\pi} \\
A_{PL}(M \sqcup \ast) \ar{r}{=} & A_{PL}(M \sqcup \ast)
\end{tikzcd}
\end{center}
that provides a quasi-isomorphism over $A_{PL}(M \sqcup \ast)$ from $h^{\ast}$ to the morphism
$$
\pi \colon A_{PL}(M \sqcup \ast) \times_{A_{PL}(\partial M \sqcup \ast)} A_{PL}(\operatorname{cone}(f_{<k})) \rightarrow A_{PL}(M \sqcup \ast)
$$
given by projection of the fiber product to the first component.
It remains to relate the morphism $\pi$ to the morphism
$$
\iota_{\overline{p}} \colon AI_{\overline{p}}(X) = A_{PL}(M \sqcup \ast) \times_{A_{PL}(\partial M \sqcup \ast)} \widehat{\tau}_{\geq k}A_{PL}(\partial M) \rightarrow A_{PL}(M \sqcup \ast)
$$
obtained as in \Cref{definition convenient rational model} from the choice of an augmented standard $k$-cotruncation
$$
(\vartheta_{\geq k}, \varepsilon) \colon \widehat{\tau}_{\geq k}A_{PL}(\partial M) \rightarrow A_{PL}(\partial M \sqcup \ast) = A_{PL}(\partial M) \oplus \mathbb{Q}
$$
of $A_{PL}(\partial M)$ (see \Cref{example standard cotruncation}).
For this purpose, we observe that the morphisms $\operatorname{incl}^{\ast} \colon A_{PL}(\operatorname{cone}(f_{<k})) \rightarrow A_{PL}(L \sqcup \ast)$ and $(\vartheta_{\geq k}, \varepsilon)$ are weakly equivalent over $A_{PL}(\partial M \sqcup \ast)$ by \Cref{proposition weak equivalence of cotruncations}.
Then, we see that the two morphisms $\pi$ and $\iota_{\overline{p}}$ are indeed weakly equivalent over $A_{PL}(M \sqcup \ast) = A_{PL}(M) \oplus \mathbb{Q}$ by applying the last part of \Cref{corollary weak equivalence of fiber products} to the diagram
\begin{center}
\begin{tikzcd}
A_{PL}(M \sqcup \ast) \times_{A_{PL}(\partial M \sqcup \ast)} A_{PL}(\operatorname{cone}(f_{<k})) \ar[r] \ar[d, "\pi"] & A_{PL}(\operatorname{cone}(f_{<k})) \ar[d, "\operatorname{incl}^{\ast}"] \\
A_{PL}(M \sqcup \ast) \ar[r, twoheadrightarrow, "\operatorname{incl}^{\ast}"] & A_{PL}(L \sqcup \ast)
\end{tikzcd}
\end{center}
and to the diagram (\ref{diagram definition of commutative model}), as well as to the identity morphisms on $A_{PL}(M \sqcup \ast)$ and $A_{PL}(L \sqcup \ast)$.

\end{proof}

\section{Proof of \Cref{main result}}\label{proof of main result}
Let $X^{n}$ be a compact $n$-dimensional topologically stratified pseudomanifold with one isolated singularity $x$ and connected link $L$.
As in \Cref{A commutative model for intersection spaces}, we fix a decomposition $X^{n} = M \cup_{\partial M} \operatorname{cone}(\partial M)$, where $M$ is a compact topological $n$-manifold with boundary $L = \partial M$.
Throughout this section, let $i \colon L = \partial M \rightarrow M$ and $j \colon (M, \emptyset) \rightarrow (M, \partial M)$ denote the inclusions.

Let $\overline{p}$ and $\overline{q}$ be complementary perversities.
That is, we have $\overline{p}(s) + \overline{q}(s) = s-2$ for all $s \in \{2, 3, \dots\}$.
Following \Cref{proposition essential diagrams}(i), we have inclusions $\iota_{\overline{p}} \colon AI_{\overline{p}}(X, x) \hookrightarrow A_{PL}(M)$ and $\iota_{\overline{p}} \colon AI_{\overline{p}}(X, x) \hookrightarrow A_{PL}(M)$ of commutative cochain complexes.

Let $X$ be $\mathbb{Q}$-oriented, which means that the regular stratum $X \setminus \{x\}$ of $X$ is equipped with a $\mathbb{Q}$-orientation.
Let $\mu \in C_{n}(M; \mathbb{Q})$ be a normalized singular chain whose image under the map $j_{\sharp} \colon C_{n}(M; \mathbb{Q}) \rightarrow C_{n}(M, \partial M; \mathbb{Q})$ represents the fundamental class in $H_{n}(M, \partial M; \mathbb{Q})$ induced by the given $\mathbb{Q}$-orientation of $X$.
As explained in \Cref{The commutative cochain algebra apl}, integration of polynomial differential forms over simplices of $\mu$ induces a linear form $\int_{\mu} \colon A_{PL}(M) \rightarrow \mathbb{Q}$.

Then, we have the following
\begin{lemma}\label{lemma existence of bilinear form on cohomology}
Multiplication in $A_{PL}(M)$ followed by integration over $\mu$ induces a bilinear form
\begin{align}\label{bilinear form main theorem}
\int_{\mu} \colon H^{r}(AI_{\overline{p}}(X, x)) \times H^{n-r}(AI_{\overline{q}}(X, x)) \rightarrow \mathbb{Q}, \quad ([\alpha], [\beta]) \mapsto \int_{\mu} \iota_{\overline{p}}(\alpha) \cdot \iota_{\overline{q}}(\beta).
\end{align}
\end{lemma}

\begin{proof}
Let us show that the bilinear form
\begin{align*}
\int_{\mu} \colon AI_{\overline{p}}^{r}(X, x) \times AI_{\overline{q}}^{n-r}(X, x) \rightarrow \mathbb{Q}, \quad (\alpha, \beta) \mapsto \int_{\mu} \iota_{\overline{p}}(\alpha) \cdot \iota_{\overline{q}}(\beta),
\end{align*}
induces a well-defined bilinear form on cohomology.
We have to show that for all closed elements $\alpha, \alpha' \in AI^{r}_{\overline{p}}(X, x)$ with $\alpha' - \alpha = d \eta$ for some $\eta \in AI^{r-1}_{\overline{p}}(X, x)$, and for all closed elements $\beta, \beta' \in AI^{n-r}_{\overline{q}}(X, x)$ with $\beta' - \beta = d \omega$ for some $\omega \in AI^{n-r-1}_{\overline{q}}(X, x)$, we have $\int_{\mu}(\alpha', \beta') = \int_{\mu}(\alpha, \beta)$.
It suffices to consider the case that $\alpha = \alpha'$ or $\beta = \beta'$.
If $\beta = \beta'$, then using Stokes' theorem (\Cref{stokes theorem}) and \Cref{lemma vanishing products},
\begin{align*}
\int_{\mu}(\alpha', \beta) - \int_{\mu}(\alpha, \beta) = \int_{\mu} d(\iota_{\overline{p}}(\eta) \cdot \iota_{\overline{q}}(\beta)) = \int_{\partial \mu} i^{\ast}(\iota_{\overline{p}}(\eta) \cdot \iota_{\overline{q}}(\beta)) = \int_{\partial \mu} y \cdot z = 0,
\end{align*}
where $y = i^{\ast}\iota_{\overline{p}}(\eta) \in \vartheta_{\geq k}(\tau_{\geq k}A_{PL}^{r-1}(L))$ and $z = i^{\ast}\iota_{\overline{q}}(\beta) \in \vartheta_{\geq l}(\tau_{\geq l}A_{PL}^{n-r}(L))$.
The proof in the case $\alpha = \alpha'$ is similar and will be omitted.
\end{proof}

It remains to show that the bilinear form of \Cref{lemma existence of bilinear form on cohomology} is nondegenerate.
For this purpose, we assemble the cohomology long exact sequences induced by the short exact sequences of the cochain complexes (see part (ii) and (iii) of \Cref{proposition essential diagrams})
\begin{align*}
0 \rightarrow AI_{\overline{p}}(X, x) \xrightarrow{\iota_{\overline{p}}} A_{PL}(M) \xrightarrow{\kappa_{\overline{p}}} A_{PL}(L) / \vartheta_{\geq k}(\tau_{\geq k}A_{PL}(L)) \rightarrow 0, \\
0 \rightarrow A_{PL}(M, \partial M) \xrightarrow{\eta_{\overline{q}}} AI_{\overline{q}}(X, x) \xrightarrow{\rho_{\overline{q}}} \tau_{\geq l}A_{PL}(L) \rightarrow 0,
\end{align*}
in a ladder diagram
\begin{center}
\begin{tikzcd}
H^{r-1}(A_{PL}(L) / \vartheta_{\geq k}(\tau_{\geq k}A_{PL}(L))) \ar{r}{\int_{\partial \mu}}[swap]{\cong} \ar{d}[swap]{\delta} & H^{n-r}(\tau_{\geq l}A_{PL}(L))^{\dagger} \ar{d}{H(\rho_{\overline{q}})^{\dagger}} \\
H^{r}(AI_{\overline{p}}(X, x)) \ar{r}{\int_{\mu}} \ar{d}[swap]{H(\iota_{\overline{p}})} & H^{n-r}(AI_{\overline{q}}(X, x))^{\dagger} \ar{d}{H(\eta_{\overline{q}})^{\dagger}} \\
H^{r}(A_{PL}(M)) \ar{r}{\int_{\mu}}[swap]{\cong} \ar{d}[swap]{H(\kappa_{\overline{p}})} & H^{n-r}(A_{PL}(M, \partial M))^{\dagger} \ar{d}{\delta^{\dagger}} \\
H^{r}(A_{PL}(L) / \vartheta_{\geq k}(\tau_{\geq k}A_{PL}(L))) \ar{r}{\int_{\partial \mu}}[swap]{\cong} & H^{n-r-1}(\tau_{\geq l}A_{PL}(L))^{\dagger},
\end{tikzcd}
\end{center}
where we also use the nondegenerate bilinear form $\int_{\partial \mu}$ of \Cref{proposition truncated poincare duality} (applied to the chain representative $\lambda = \partial \mu\in C_{n-1}(L)$ of the fundamental class of $L = \partial M$ that is induced by the $\mathbb{Q}$-orientation of $M$), the nondegenerate bilinear form $\int_{\mu}$ of \Cref{poincare lefschetz duality theorem}, and the bilinear form $\int_{\mu}$ of \Cref{lemma existence of bilinear form on cohomology}.

In view of the five lemma, to show that the bilinear form (\ref{bilinear form main theorem}) is nondegenerate it suffices to show that in the ladder diagram above, the top square (TS), the middle square (MS), and the bottom square (BS) commute (up to sign).

Let us show that the top square (TS) commutes.
We have to show that, for all closed elements $\alpha \in A_{PL}^{r-1}(L) / \vartheta_{\geq k}(\tau_{\geq k}A_{PL}^{r-1}(L))$ and all closed elements $\beta \in AI_{\overline{q}}^{n-r}(X, x)$, it holds that
\begin{align}\label{equation TS}
((\int_{\mu} \circ\delta)([\alpha]))([\beta]) = (H(\rho_{\overline{q}})^{\dagger} \circ\int_{\partial \mu})([\alpha]))([\beta]).
\end{align}
Let us consider the left hand side of equation (\ref{equation TS}).
By construction of $\delta$, we can write $\delta([\alpha]) = [y]$, where $y \in AI^{r}_{\overline{p}}(X, x)$ satisfies $\iota_{\overline{p}}(y) = dx$ for some $x \in A_{PL}^{r-1}(M)$ with $\kappa_{\overline{p}}(x) = \alpha$.
Thus, the left hand side of the equation is given by
\begin{align*}
((\int_{\mu} \circ\delta)([\alpha]))([\beta]) = (\int_{\mu}([y]))([\beta]) = \int_{\mu} \iota_{\overline{p}}(y) \cdot \iota_{\overline{q}}(\beta) = \int_{\mu} dx \cdot \iota_{\overline{q}}(\beta).
\end{align*}
Using Stokes' theorem (\Cref{stokes theorem}) and the fact that $\beta$ is closed, we obtain
\begin{align*}
\int_{\mu} dx \cdot \iota_{\overline{q}}(\beta) = \int_{\mu} d(x \cdot \iota_{\overline{q}}(\beta)) = \int_{\partial\mu} i^{\ast}(x \cdot \iota_{\overline{q}}(\beta)).
\end{align*}
Since $\kappa_{\overline{p}} = \pi_{\geq k} \circ i^{\ast}$ (\Cref{proposition essential diagrams}(iii)), the right hand side of (\ref{equation TS}) becomes
\begin{align*}
(H(\rho_{\overline{q}})^{\dagger} \circ\int_{\partial\mu})([\alpha]))([\beta]) = (\int_{\partial\mu}([\alpha]))([\rho_{\overline{q}}(\beta)]) = \int_{\partial\mu} i^{\ast}(x) \cdot \vartheta_{\geq l}(\rho_{\overline{q}}(\beta)).
\end{align*}
Noting that $i^{\ast} \circ \iota_{\overline{q}} = \vartheta_{\geq l} \circ \rho_{\overline{q}}$ (\Cref{proposition essential diagrams}(i)), we conclude that (TS) commutes.

Next, we show that the middle square (MS) commutes.
We have to show that for all closed elements $\alpha \in AI^{r}_{\overline{p}}(X, x)$ and all closed elements $\beta \in A_{PL}^{n-r}(M, \partial M)$, it holds that
\begin{align}\label{equation MS}
((\int_{\mu} \circ H(\iota_{\overline{p}}))([\alpha]))([\beta]) = ((H(\eta_{\overline{q}})^{\dagger} \circ\int_{\mu})([\alpha]))([\beta]).
\end{align}
Let us consider the left hand side of equation (\ref{equation MS}), which is
$$
((\int_{\mu} \circ H(\iota_{\overline{p}}))([\alpha]))([\beta]) = (\int_{\mu}([\iota_{\overline{p}}(\alpha)]))([\beta]) = \int_{\mu} \iota_{\overline{p}}(\alpha) \cdot j^{\ast}(\beta).
$$
On the other hand, the right hand side of (\ref{equation MS}) is given by
$$
((H(\eta_{\overline{q}})^{\dagger} \circ\int_{\mu})([\alpha]))([\beta]) = (\int_{\mu}([\alpha]))([\eta_{\overline{q}}(\beta)]) = \int_{\mu} \iota_{\overline{p}}(\alpha) \cdot \iota_{\overline{q}}(\eta_{\overline{q}}(\beta)).
$$
Noting that $j^{\ast} = \iota_{\overline{q}} \circ \eta_{\overline{q}}$ (\Cref{proposition essential diagrams}(ii)), we conclude that (MS) commutes.

Finally, let us show that the bottom square (BS) commutes (up to sign).
We have to show that for all closed elements $\alpha \in A^{r}_{PL}(M)$ and all closed elements $\beta \in \tau_{\geq l}A^{n-r-1}_{PL}(L)$, it holds that
\begin{align}\label{equation BS}
((\int_{\partial \mu} \circ H(\kappa_{\overline{p}}))([\alpha]))([\beta]) = ((\delta^{\dagger} \circ\int_{\mu})([\alpha]))([\beta]).
\end{align}
Since $\kappa_{\overline{p}} = \pi_{\geq k} \circ i^{\ast}$ (\Cref{proposition essential diagrams}(iii)), the left hand side of equation (\ref{equation BS}) becomes
$$
((\int_{\partial \mu} \circ H(\kappa_{\overline{p}})([\alpha]))([\beta]) = (\int_{\partial \mu}([\pi_{\geq k}(i^{\ast}(\alpha))]))([\beta]) = \int_{\partial\mu} i^{\ast}(\alpha) \cdot \vartheta_{\geq l}(\beta).
$$
Let us compare this to the right hand side of (\ref{equation BS}).
By construction of $\delta$, we can write $\delta([\beta]) = [y]$, where $y \in A_{PL}^{n-r}(M, \partial M)$ satisfies $\eta_{\overline{q}}(y) = dx$ for some $x \in AI_{\overline{q}}^{n-r-1}(X, x)$ with $\rho_{\overline{q}}(x) = \beta$.
Since $j^{\ast} = \iota_{\overline{q}} \circ \eta_{\overline{q}}$ (\Cref{proposition essential diagrams}(ii)), the right hand side of (\ref{equation BS}) is given by
$$
((\delta^{\dagger} \circ\int_{\mu})([\alpha]))([\beta]) = (\int_{\mu}([\alpha]))([y]) = \int_{\mu} \alpha \cdot j^{\ast}(y) = \int_{\mu} \alpha \cdot \iota_{\overline{q}}(dx).
$$
Using Stokes' theorem (\Cref{stokes theorem}) and the fact that $\alpha$ is closed, we obtain
$$
\int_{\mu} \alpha \cdot \iota_{\overline{q}}(dx) = \pm \int_{\mu} d(\alpha \cdot \iota_{\overline{q}}(x)) = \pm \int_{\partial\mu} i^{\ast}(\alpha \cdot \iota_{\overline{q}}(x)).
$$
Noting that $i^{\ast} \circ \iota_{\overline{q}} = \vartheta_{\geq l} \circ \rho_{\overline{q}}$ (\Cref{proposition essential diagrams}(i)), we conclude that (BS) commutes (up to sign).

This completes the proof of \Cref{main result}.

\section{Smooth differential forms}\label{Smooth differential forms}
In this section, we work over the ground field $\mathbb{R}$.
We also assume that the pseudomanifold $X$ considered at the beginning of \Cref{A commutative model for intersection spaces} has a Thom-Mather $C^{\infty}$-stratification.
Since $X$ has one isolated singularity $x$, the only condition is that the regular stratum $X \setminus \{x\}$ of $X$ is equipped with a smooth structure.



Let $\Omega^{\ast}(M)$ denote the commutative cochain algebra of smooth differential forms on the smooth manifold $M \subset X \setminus \{x\}$ with boundary $\partial M$.
One of the main results of \cite{ban4} is a de Rham description of the intersection space cohomology $\widetilde{H}^{\ast}(I^{\overline{p}}X; \mathbb{R})$ by means of the subcomplex
$$
\Omega I_{\overline{p}}^{\ast}(M) = \{\omega \in \Omega^{\ast}(M); \; \omega|_{\partial M} \in \tau_{\geq k} \Omega^{\ast}(\partial M)\} \subset \Omega^{\ast}(M)
$$
determined by the choice of a cotruncation cochain complex $\tau_{\geq k} \Omega^{\ast}(\partial M)$ for the cutoff degree $k = k(\overline{p})$ as in \Cref{example standard cotruncation}.
Banagl's de Rham isomorphism
$$
\widetilde{H}^{\ast}(I^{\overline{p}}X; \mathbb{R}) \cong H^{\ast}(\Omega I_{\overline{p}}^{\ast}(M))$$
is induced by integration of differential forms over smooth singular simplices, and relies on a partial smoothing technique.
In view of \Cref{main result B}, we have the following strengthening of Banagl's de Rham description.


\begin{theorem}\label{de rham theorem}
The morphism $\iota_{\overline{p}} \colon AI_{\overline{p}}(X; \mathbb{R}) \rightarrow A_{PL}(M \sqcup \ast; \mathbb{R})$ over $\mathbb{R}$ constructed as in \Cref{definition convenient rational model} and the morphism $\Omega I^{\ast}_{\overline{p}}(M) \oplus \mathbb{R} \rightarrow \Omega^{\ast}(M) \oplus \mathbb{R}$ over $\mathbb{R}$ given by $(\alpha, t) \mapsto (\alpha + t, t)$ (and with augmentations given by projection to the second component) are weakly equivalent over $\mathbb{R}$.
%
\end{theorem}

\begin{proof}
A smooth singular $k$-simplex on a smooth manifold $N$ with boundary is a continuous map $\Delta^{k} \rightarrow N$ that extends to a smooth map $U \rightarrow \widetilde{N}$, where $U$ is an open neighborhood of the standard simplex $\Delta ^{k} \subset \mathbb{R}^{k+1}$, and $\widetilde{N}$ is obtained from $N$ by gluing an outward collar to $\partial N$.
Let $S^{\infty}_{\ast}(N)$ denote the simplicial set of smooth singular simplices on $N$ (compare \S 11(c) in \cite{fht}).
Moreover, let $A_{DR}$ denote the simplicial cochain algebra of real $C^{\infty}$ differential forms (see \S 11(c) in \cite{fht}).

According to Theorem 11.4 in \cite{fht} (whose proof is also applicable to manifolds with boundary by taking into account charts at boundary points modeled on upper Euclidean half space), the inclusion $S^{\infty}_{\ast}(N) \rightarrow S_{\ast}(N)$ into the simplicial set $S_{\ast}(N)$ of all singular simplices on $N$ induces a natural quasi-isomorphism
$$
\gamma_{N} \colon A_{PL}(N; \mathbb{R}) = A_{PL}(S_{\ast}(N); \mathbb{R}) \xrightarrow{\simeq} A_{PL}(S_{\ast}^{\infty}(N); \mathbb{R}).
$$
Moreover, by the same theorem, the inclusion $A_{PL}(\quad; \mathbb{R}) \rightarrow A_{DR}$ of simplicial cochain algebras induces a natural quasi-isomorphism
$$
\beta_{N} \colon A_{PL}(S_{\ast}^{\infty}(N); \mathbb{R}) \xrightarrow{\simeq} A_{DR}(S_{\ast}^{\infty}(N)).
$$
We apply \Cref{corollary weak equivalence of fiber products} (in its version over the ground field $\mathbb{R}$) to the quasi-isomorphism $(\alpha, \gamma) \colon f \xrightarrow{\simeq} f'$ over $\mathbb{R}$ given by the commutative diagram of morphisms over $\mathbb{R}$
\begin{center}
\begin{tikzcd}
A_{PL}(M \sqcup \ast; \mathbb{R}) \ar{d}{\alpha = \beta_{M \sqcup \ast} \circ \gamma_{M \sqcup \ast}}[swap]{\simeq} \ar{r}{f} & A_{PL}(\partial M \sqcup \ast; \mathbb{R}) \ar{d}{\gamma = \beta_{\partial M \sqcup \ast} \circ \gamma_{\partial M \sqcup \ast}}[swap]{\simeq}  \\
A_{DR}(S^{\infty}_{\ast}(M \sqcup \ast)) \ar{r}{f'} & A_{DR}(S^{\infty}_{\ast}(\partial M \sqcup \ast)),
\end{tikzcd}
\end{center}
and to chosen augmented standard $k$-cotruncations (see \Cref{example standard cotruncation})
\begin{align*}
g \colon \widehat{\tau}_{\geq k}A_{PL}(\partial M; \mathbb{R}) &\rightarrow A_{PL}(\partial M \sqcup \ast; \mathbb{R})  = A_{PL}(\partial M; \mathbb{R}) \oplus \mathbb{R}, \\
g' \colon \widehat{\tau}_{\geq k}A_{DR}(S^{\infty}_{\ast}(\partial M)) &\rightarrow A_{DR}(S^{\infty}_{\ast}(\partial M \sqcup \ast)) = A_{DR}(S^{\infty}_{\ast}(\partial M)) \oplus \mathbb{R}
\end{align*}
of $A_{PL}(\partial M; \mathbb{R})$ and $A_{DR}(S^{\infty}_{\ast}(\partial M))$, respectively, to obtain a weak equivalence over $\mathbb{R}$ between the morphism $\iota_{\overline{p}} \colon AI_{\overline{p}}(X; \mathbb{R}) \rightarrow A_{PL}(M \sqcup \ast; \mathbb{R})$ and the morphism
$$
\kappa_{\overline{p}} \colon A_{DR}(S^{\infty}_{\ast}(M \sqcup \ast)) \times_{A_{DR}(S^{\infty}_{\ast}(\partial M \sqcup \ast))} \widehat{\tau}_{\geq k}A_{DR}(S^{\infty}_{\ast}(\partial M)) \rightarrow A_{DR}(S^{\infty}_{\ast}(M \sqcup \ast))
$$
given by projection of the fiber product to the first component.
(For this purpose, we note that our morphisms $f$ and $f'$ in the above diagram are indeed surjective because $A_{PL}(\quad; \mathbb{R})$ is extendable by Lemma 10.7(iii) in \cite{fht}, and $A_{DR}$ is extendable by Lemma 11.3(iii) in \cite{fht}.
Moreover, we observe that our morphisms $g'$ and $\gamma \circ g$ are weakly equivalent over $A_{DR}(S^{\infty}_{\ast}(\partial M \sqcup \ast))$ by \Cref{proposition weak equivalence of cotruncations}.)

By Theorem 11.4 in \cite{fht}, there is also a natural quasi-isomorphism
$$
\alpha_{N} \colon \Omega^{\ast}(N) \xrightarrow{\simeq} A_{DR}(S_{\ast}^{\infty}(N)).
$$
We apply \Cref{corollary weak equivalence of fiber products} (in its version over the ground field $\mathbb{R}$) to the quasi-isomorphism $(\alpha, \gamma) \colon f \xrightarrow{\simeq} f'$ over $\mathbb{R}$ given by the commutative diagram of morphisms over $\mathbb{R}$
\begin{center}
\begin{tikzcd}
\Omega^{\ast}(M \sqcup \ast) \ar{d}{\alpha = \alpha_{M \sqcup \ast}}[swap]{\simeq} \ar{r}{f} & \Omega^{\ast}(\partial M \sqcup \ast) \ar{d}{\gamma = \alpha_{\partial M \sqcup \ast}}[swap]{\simeq}  \\
A_{DR}(S^{\infty}_{\ast}(M \sqcup \ast)) \ar{r}{f'} & A_{DR}(S^{\infty}_{\ast}(\partial M \sqcup \ast))
\end{tikzcd}
\end{center}
and to the previously fixed augmented standard $k$-cotruncations
\begin{align*}
\widehat{\tau}_{\geq k} \Omega^{\ast}(\partial M) &\rightarrow \Omega^{\ast}(\partial M \sqcup \ast) = \Omega^{\ast}(\partial M) \oplus \mathbb{R}, \\
\widehat{\tau}_{\geq k}A_{DR}(S^{\infty}_{\ast}(\partial M)) &\rightarrow A_{DR}(S^{\infty}_{\ast}(\partial M \sqcup \ast)) = A_{DR}(S^{\infty}_{\ast}(\partial M)) \oplus \mathbb{R}
\end{align*}
of $\Omega^{\ast}(\partial M \sqcup \ast)$ and $A_{DR}(S^{\infty}_{\ast}(\partial M))$, taking the roles of $g$ and $g'$, respectively, to obtain a weak equivalence over $\mathbb{R}$ between the morphism
$$
\lambda_{\overline{p}} \colon \Omega^{\ast}(M \sqcup \ast) \times_{\Omega^{\ast}(\partial M \sqcup \ast)} \widehat{\tau}_{\geq k} \Omega^{\ast}(\partial M) \rightarrow \Omega^{\ast}(M \sqcup \ast)
$$
given by projection of the fiber product to the first component and the morphism $\kappa_{\overline{p}}$.
(For this purpose, we observe that our morphisms $g'$ and $\gamma \circ g$ are indeed weakly equivalent over $A_{DR}(S^{\infty}_{\ast}(\partial M \sqcup \ast))$ by \Cref{proposition weak equivalence of cotruncations}.)

Finally, we observe that the morphism $\Omega I^{\ast}_{\overline{p}}(M) \oplus \mathbb{R} \rightarrow \Omega^{\ast}(M) \oplus \mathbb{R}$ over $\mathbb{R}$ given by $(\alpha, t) \mapsto (\alpha + t, t)$ corresponds to the morphism $\lambda_{\overline{p}}$ under the composition of isomorphisms
\begin{center}
\begin{tikzcd}
\Omega I^{\ast}_{\overline{p}}(M) \oplus \mathbb{R} \ar{d}{\cong}[swap]{(\omega, t) \mapsto ((\omega, \omega|_{\partial M}), t)} \ar{r}{\cong} & \Omega^{\ast}(M \sqcup \ast) \times_{\Omega^{\ast}(\partial M \sqcup \ast)} \widehat{\tau}_{\geq k} \Omega^{\ast}(\partial M) \\
\left(\Omega^{\ast}(M) \times_{\Omega^{\ast}(\partial M)} \tau_{\geq k} \Omega^{\ast}(\partial M)\right) \oplus \mathbb{R} \ar{r}{\Psi}[swap]{\cong} & \Omega^{\ast}(M) \oplus \mathbb{R} \times_{\Omega^{\ast}(\partial M) \oplus \mathbb{R}}\tau_{\geq k} \Omega^{\ast}(\partial M) \oplus \mathbb{R} \ar{u}[swap]{=},
\end{tikzcd}
\end{center}
where $\Psi((\omega, \eta), t) = ((\omega + t, t), (\eta, t))$.
\end{proof}

\subsection*{Acknowledgements}
This work was written while the author was a JSPS International Research Fellow (Postdoctoral Fellowships for Research in Japan (Standard)).

\bibliographystyle{amsplain}

\end{document}